\documentclass[reqno,a4paper,11pt]{amsart}
\usepackage[T1]{fontenc}
\usepackage[utf8]{inputenc}
\usepackage[english]{babel}
\usepackage{url}
\usepackage[
colorlinks=true, linkcolor = black, urlcolor=black, citecolor=blue, anchorcolor=blue]{hyperref}
\usepackage{comment}
\usepackage{tcolorbox}

\usepackage{cancel}

\newcommand{\crosses}[1]{%
	\ifcase#1\relax
	\or
	\rslash\or
	\rslash\mskip-5.5mu\rslash\or
	\rslash\mskip-5.5mu\rslash\mskip-5.5mu\rslash%
	\fi
}
\newcommand{\rslash}{\raisebox{.15ex}{/}}
\usepackage{xcolor,color}
\usepackage{geometry}
\usepackage[all, cmtip]{xy}
\usepackage{enumitem}
\usepackage{booktabs}
\usepackage{slashed}
\usepackage{bbm}
\usepackage{cancel}

\usepackage{cleveref}
\usepackage{tikz}
\usepackage{tikz-cd}
\tikzcdset{arrow style=tikz, diagrams={>=stealth}}

\usepackage{xargs}  
\usepackage{soul}

\usepackage{amsmath}
\usepackage{amssymb,graphicx}
\usepackage{amsthm}
\usepackage{latexsym}
\usepackage{amsfonts}
\usepackage{bbm}
\usepackage{mathrsfs}
\usepackage{cases}

\usepackage{etex}

\calclayout

\numberwithin{equation}{section}

\theoremstyle{plain}
\newtheorem{lemma}{Lemma}[section]
\newtheorem{proposition}[lemma]{Proposition}
\newtheorem{proposition/definition}[lemma]{Proposition/Definition}
\newtheorem{theorem}[lemma]{Theorem}
\newtheorem{corollary}[lemma]{Corollary}

\newtheorem*{theorem*}{Theorem}

\theoremstyle{definition}
\newtheorem{definition}[lemma]{Definition}
\newtheorem{remark}[lemma]{Remark}
\newtheorem{example}[lemma]{Example}

\DeclareRobustCommand{\SkipTocEntry}[5]{}

\newcommand{\bbnabla}{\nabla \hspace{-4.5pt} \nabla}

\renewcommand{\theta}{\vartheta}

\allowdisplaybreaks

\title{Multiplicative Connections and Their Lie Theory}

\makeatletter
\def\author@andify{%
	\nxandlist {\unskip ,\penalty-1 \space\ignorespaces}%
	{\unskip {} \@@and~}%
	{\unskip \penalty-2 \space \@@and~}%
}
\makeatother

\author[F.~Pugliese]{Fabrizio Pugliese}
\address{DipMat, Universit\`a degli Studi di Salerno, via Giovanni Paolo II n${}^{\circ}$123, 84084 Fisciano (SA), Italy.}
\email{\href{mailto:fpugliese@unisa.it}{fpugliese@unisa.it}}

\author[G.~Sparano]{Giovanni Sparano}
\address{DipMat, Universit\`a degli Studi di Salerno, via Giovanni Paolo II n${}^{\circ}$123, 84084 Fisciano (SA), Italy.}
\email{\href{mailto:sparano@unisa.it}{sparano@unisa.it}}

\author[L.~Vitagliano]{Luca Vitagliano}
\address{DipMat, Universit\`a degli Studi di Salerno, via Giovanni Paolo II n${}^{\circ}$123, 84084 Fisciano (SA), Italy.}
\email{\href{mailto:lvitagliano@unisa.it}{lvitagliano@unisa.it}}



\begin{document}

\begin{abstract}
We define and study \emph{multiplicative connections} in the tangent bundle of a Lie groupoid. Multiplicative connections are linear connections satisfying an appropriate compatibility with the groupoid structure. Our definition is \emph{natural} in the sense that a linear connection on a Lie groupoid is multiplicative if and only if its torsion is a multiplicative tensor in the sense of Bursztyn-Drummond \cite{BD2019} and its geodesic spray is a multiplicative vector field. We identify the obstruction to the existence of a multiplicative connection. We also discuss the infinitesimal version of multiplicative connections in the tangent bundle, that we call \emph{infinitesimally multiplicative (IM) connections} and we prove an integration theorem for IM connections. Finally, we present a few toy examples. 
\end{abstract}

\maketitle

\tableofcontents

\section{Introduction}

Lie groupoids are important objects in Differential Geometry, particularly because they encode symmetries of geometric structures. Additionally, certain singular spaces, called \emph{differentiable stacks}, can be seen as Morita equivalence classes of Lie groupoids \cite{BX2011}. Examples of differentiable stacks are orbifolds, leaf spaces of foliations, and orbit spaces of group actions. The geometry of a differentiable stack is encoded in the so called \emph{transverse geometry} of a Lie groupoid representing it \cite{dH2013}, and geometric structures on a differentiable stack are often represented by cohomology classes of appropriate complexes attached to the Lie groupoid. For instance, a vector field on a differentiable stack $X$ represented by a Lie groupoid $\mathcal G$ is represented by a cocycle in the deformation complex of $\mathcal G$ \cite{H2009, OW2019, CMS2018}. Similarly, symplectic structures on $X$, the so called \emph{shifted symplectic structures}, are represented by cocycles in the Bott-Shulman-Stasheff complex of $\mathcal G$, satisfying an appropriate non-degeneracy condition up to homotopy \cite{G2014}.

From a somehow different point of view, Lie groupoids sometimes come with geometric structures compatible with the groupoid structure. Such structures are often called \emph{multiplicative} in the current literature. For instance a multiplicative vector field on a Lie groupoid is a vector field generating a flow by local groupoid automorphisms \cite{MX1998}, or complex Lie groupoids are (real) Lie groupoids equipped with a multiplicative complex structure \cite{LSX2009}. Similarly, symplectic groupoids \cite{K1987, W1987} are Lie groupoids equipped with a multiplicative symplectic structure and so on. In the last 30 years several different kinds of multiplicative structures on Lie groupoids have been defined and studied, including multiplicative Poisson structures \cite{W1987, LSX2007}, multiplicative contact structures \cite{KS1993, L1993, BSTV2019}, multiplicative foliations \cite{JO2014}, multiplicative tensors \cite{BD2019}, and many many others (see \cite{KS2016} for a survey). Multiplicative structures on Lie groupoids are interesting at least for two reasons.

The first reason is that they are usually the \emph{global counterparts} of interesting (infinitesimal) geometric structures on the space of objects. Namely, as Lie groups have their infinitesimal counterparts in Lie algebras, Lie groupoids have their infinitesimal counterparts in Lie algebroids. Multiplicative structures on a Lie groupoid $\mathcal G \rightrightarrows M$ can be linearized around the manifold of objects $M$ to produce structures on the associated Lie algebroid $A$. Such infinitesimal versions of multiplicative structures are often called \emph{infinitesimally multiplicative} (\emph{IM} for short, in the following). In their turn IM structures on $A$ encode structures on $M$. For instance, an IM symplectic structure encodes a Poisson structure on $M$, an IM contact structure is a Jacobi structure on $M$, an IM complex structure is a complex structure on $M$ plus a holomorphic Lie algebroid, etc. And one can efficiently study the infinitesimal structures studying their global counterparts.

The second reason why multiplicative structures are important is related to differentiable stacks. Consider a differentiable stack $X$ represented by a Lie groupoid $\mathcal G$. As we have already mentioned, geometric structures on $X$ are often represented by cocycles in appropriate cochain complexes attached to $\mathcal G$. It turns out that multiplicative structures are then cocycles of a specific kind. For instance, multiplicative (pre-)symplectic structures are cocycles of a special kind in the Bott-Shulman-Stasheff complex. 
In this sense, the theory of \emph{multiplicative structures} on Lie groupoids is a ``simplified'' version of the theory of geometric structures on differentiable stacks.

This paper belongs to a series devoted to explore differential calculus on differentiable stacks. One of our aims is defining the building blocks of such differential calculus, namely differential operators and jets on differentiable stacks. We began in \cite{ETV2019, LV2019, LV20??} identifying a cochain complex whose cohomologies can be interpreted as derivations of a vector bundle over a differentiable stack and we wish to continue defining and studying linear connections on differentiable stacks. As a warm up, in this paper, we define multiplicative connections in the tangent bundle of a Lie groupoid. We hope we will be able to provide a definition of linear connections over a differentiable stack at a later stage of this project.

The paper is organized as follows. In Section \ref{Sec:Lin_Conn} we recall the fundamentals about linear connections. In Section \ref{Sec:Mult_Conn} we define multiplicative connections and discuss their basic properties. Our definition is inspired by that of multiplicative vector valued forms by Bursztyn-Drummond \cite{BD2013} (more precisely the equivalent definition provided by \cite[Lemma 4.1]{BD2013}). Specifically, let $\mathcal G \rightrightarrows M$ be a Lie groupoid with source, target and multiplication denoted $s,t,m$, then a linear connection $\nabla$ in $T\mathcal G \to \mathcal G$ is \emph{multiplicative} if it is $(s,t)$-projectable (i.e. there is a connection $\nabla^M$ in $TM \to M$ such that $\nabla, \nabla^M$ are both $s$ and $t$-related), in which case the product connection $\nabla \times \nabla$ in $T (\mathcal G \times \mathcal G) \to \mathcal G \times \mathcal G$ restricts to the manifold $\mathcal G^{(2)}$ of composable arrows, and, additionally, the restriction $(\nabla \times \nabla)|_{\mathcal G^{(2)}}$ and $\nabla$ are $m$-related. This definition is not arbitrary. It is natural in the sense that a linear connection in $T\mathcal G$ is multiplicative iff its torsion is a multiplicative vector valued $2$-form and its geodesic spray is a multiplicative vector field on $T \mathcal G$ (see Proposition \ref{prop:char_mult_conn}). Not all Lie groupoids admit a multiplicative connection in their tangent bundle, and we also identify the obstruction to the existence, showing that a multiplicative connection exists iff a certain (characteristic) cohomology class in an appropriate cochain complex does vanish (Theorem \ref{theor:obst_class}). Our complex is a close relative of the deformation complex of Crainic-Mestre-Struchiner \cite{CMS2018}, and multiplicative $(2,1)$-tensors are exactly $1$-cocycles therein. Additionally, the obstruction class is an infinitesimal version of the Atiyah class of a dg-vector bundle \cite{MSX2015}. In Section \ref{Sec:IM_Conn} we present the \emph{infinitesimal theory}. Let $A \Rightarrow M$ be a Lie algebroid, a linear connection in $TA \to A$ is \emph{infinitesimally multiplicative} (IM for short) if it is compatible with the Lie algebroid structure in an appropriate sense.  
In the same spirit as \cite{BD2019}, we define IM connections in terms of appropriate components (Definition \ref{def:IM_conn}). Our definition is natural in the sense that a linear connection in $T A$ is IM iff its torsion is an IM vector valued $2$-form and its geodesic spray is an IM vector field on $T A$ wrt to the tangent prolongation Lie algebroid structure $TA \Rightarrow TM$ (see Theorem \ref{theor:IM_spray}). An even stronger motivation is provided by the following theorem, which is also our main result (see Theorem \ref{theor:Lie_theory} for a more detailed statement):
\begin{theorem}
 Let $\mathcal G \rightrightarrows M$ be a Lie groupoid, let $A \Rightarrow M$ be its Lie algebroid, and let $\nabla$ be a multiplicative connection in $T\mathcal G \to \mathcal G$. Then $\nabla$ determines an IM connection $\nabla^{\mathrm{lin}}$ in $TA \to A$. If $\mathcal G$ is source simply connected, the assignment $\nabla \mapsto \nabla^{\mathrm{lin}}$ establishes a bijection between multiplicative connections in $T \mathcal G \to \mathcal G$ and IM connections in $TA \to A$.
\end{theorem}

In Section \ref{Sec:Examples} we discuss a few examples. It should be noticed, in this respect, that multiplicative (and IM) connections are not abundant objects, and the existence of a multiplicative connection on the tangent bundle of a Lie groupoid $\mathcal G$ imposes severe conditions on $\mathcal G$. For instance $\mathcal G$ is necessarily regular and the isotropy Lie algebra bundle of $\mathcal G$ is necessarily abelian. As a consequence, it is not easy to find interesting examples. This might disappoint some readers. Recall however that multiplicative structures do often impose constraints on the Lie groupoid. For instance a Lie algebroid with an IM symplectic structure is necessarily the cotangent Lie algebroid of a Poisson manifold. Additionally, we stress again that our ultimate goal is actually defining and studying connections on differentiable stacks, which we think should be more abundant. In any case, we show in Section \ref{Sec:Examples} that orbifold groupoids might possess a multiplicative connection in which case the latter induce a connection in the associated orbifold, confirming the idea that there might be a relation between multiplicative connections and the should-be notion of connection on a differentiable stack. Somehow similarly, submersion groupoids possess multiplicative connections and they induce connections on their orbit space (i.e. the base manifold of the submersion). We also show that every transitive Lie algebroid with abelian isotropy possesses an IM connection in its tangent bundle. Finally, adopting appropriate simplifying ansatzes we construct a family of examples of IM connections in the tangent bundle of a Lie algebroid whose anchor in neither injective nor surjective. 

The paper is completed by two appendix. Appendix \ref{app:B} contains a proof of Theorem \ref{theor:obst_class} on the \emph{Atiyah class} obstructing the existence of multiplicative connections. Appendix \ref{app:A} contains both technical and conceptual material. Namely we provide a proof of the characterization of IM connections in terms of their torsion and geodesic spray using graded geometry techniques. To do this we first show that a Lie algebroid with an IM connection is equivalent to a degree $1$ non-negatively graded $Q$-manifold with a connection compatible with the $Q$-structure. This is not surprising and it is a manifestation of a general phenomenon: Lie algebroids are equivalent to degree $1$ non-negatively graded $Q$-manifolds, and a Lie algebroid with an IM structure of a certain kind is often equivalent to a (degree $1$ non-negatively graded) $Q$-manifold with a structure of the same kind which is additionally compatible with the homological vector field $Q$ in an appropriate sense. For instance, Lie algebroids with an IM symplectic (resp., contact) structure are the same as degree $1$ symplectic (resp., contact) $\mathbb N Q$-manifolds (see, e.g. \cite{R2002,G2013,M2013}, see also \cite{V2016}). We believe that the definition of an IM connection (and other IM structures) in terms of $Q$-manifolds should be considered as the most fundamental one. 

We conclude this introduction remarking that many definitions and results of the present paper do actually pass to the case of linear connections in arbitrary VB groupoids, i.e. vector bundle objects in the category of Lie groupoids \cite{mackenzie, BCH2016, GSM2017}. Our emphasis however is usually on the torsion and the geodesic spray, and those notions do not exist for a linear connection in an arbitrary vector bundle. So a more general theory will require a different strategy in some of the proofs.

\bigskip

\textbf{Notation.} We denote by $\mathcal G \rightrightarrows M$ a Lie groupoid with $\mathcal G$ its space of arrows, and $M$ its space of objects. We denote by $A \Rightarrow M$ a Lie algebroid with $A \to M$ its underlying vector bundle. We usually denote simply by $[-,-]_A$ the Lie bracket on sections of $A$ and by $\rho_A : A \to TM$ the anchor. Given a surjective submersion $\pi : M \to B$ we denote by $T^\pi M := \ker d\pi$ the vertical tangent bundle wrt $\pi$, and by $\mathfrak X^\pi (M)$ the $\pi$-vertical vector fields on $M$, i.e. the sections of $T^\pi M$. If $E \to M$ is a vector bundle, we denote by $DE \to M$ the \emph{gauge algebroid} of $E$, i.e. the vector bundle whose sections are derivations of the $C^\infty (M)$-module $\Gamma (E)$. So a linear connection $\nabla$ in $E$ is a vector bundle map $TM \to DE$ splitting the \emph{symbol map}, i.e. the canonical projection $\sigma : DE \to TM$.

We assume the reader is familiar with Lie groupoids and Lie algebroids, including the tangent prolongation $ T\mathcal G  \rightrightarrows TM$ (resp., $TA \Rightarrow TM$) of a Lie groupoid $\mathcal G \rightrightarrows M$ (resp., a Lie algebroid $A \Rightarrow M$). Our main reference for this material is the book by K. Mackenzie \cite{mackenzie}. Additionally, the appendix requires some familiarity with graded geometry, for which the reader might consult \cite{M2006}.

\section{Linear Connections and Smooth Maps}\label{Sec:Lin_Conn}

In this short section we fix our notation and conventions about linear connections. We also recollect here some facts about linear connections and smooth maps that will be used throughout the paper and for which it is difficult to provide references. 

Let $\pi : E \to M$ be a vector bundle and let $\nabla$ be a linear connection (now on referred to simply as a connection) in $E$. 
We will denote by $H^\nabla \subseteq TE$ the associated horizontal distribution, and by $\theta^\nabla : TE \to T^\pi E$ the \emph{vertical projection} (determined by $\nabla$), i.e. the projection with kernel $H^\nabla$. We denote with the same symbol $\nabla$ the connection induced on the tensor algebra of $E$ (and its dual).

In this paper we will mainly consider connections in the tangent bundle $TM$ (although connections in different vector bundles will also pop up). Any such connection $\nabla$ is completely determined by its torsion $T^\nabla$ and its \emph{symmetric part} $\nabla^{\mathrm{sym}} = \nabla - \frac{1}{2} T^\nabla$. 
In its turn, a symmetric connection $\nabla$ is completely determined by its \emph{geodesic spray} $Z^\nabla \in \mathfrak X (TM)$, and we will often make use of this.

We will often need to compare two connections on two manifolds $M, N$ via a smooth map $F: M \to N$. Let us recall how to do that. We begin recalling how to compare tensors. We will only need $(p,1)$-tensors. So let $M$ and $N$ be smooth manifolds, and let $F : M \to N$ be a smooth map. We can use $F$ to compare given $(p,1)$-tensors $\mathcal T^M \in \Gamma ((T^\ast M)^{\otimes p} \otimes TM)$, $\mathcal T^N \in \Gamma ((T^\ast N)^{\otimes p} \otimes TN)$. Namely, we say that $\mathcal T^M$ and $\mathcal T^N$ are \emph{$F$-related} if, for any $x \in M$, and any $v_1, \ldots, v_p \in T_x M$ we have
\begin{equation}\label{eq:T,S,F-rel}
d_x F\left(\mathcal T^M (v_1, \ldots, v_p)\right) = \mathcal T^N (d_x F (v_1), \ldots, d_x F (v_p)),
\end{equation}
where we interpreted (as we will sometimes do) a $(p,1)$-tensor $\mathcal T^M$ on $M$ as a vector bundle map $\mathcal T^M : (TM)^{\otimes p} \to TM$. 
Let $\mathcal T^M, \mathcal T^N$ be tensors as above. There are various equivalent ways to express their $F$-relatedness (according to how we interpret tensors). For instance, if we interpret $\mathcal T^M$ as a map $\mathcal T^M : T^\ast M \to (T^\ast M)^{\otimes p}$ (likewise for $\mathcal T^N$), then $\mathcal T^M, \mathcal T^N$ are $F$-related iff $\mathcal T^M ( F^\ast \theta) = F^\ast ( \mathcal T^N \theta)$ for all $\theta \in \Omega^1 (N)$.

Similarly, we can use $F$ to compare given connections $\nabla^M, \nabla^N$ in $TM$ and $TN$ respectively. Namely, we say that $\nabla^M$ and $\nabla^N$ are \emph{$F$-related} if whenever $X_1, X_2 \in \mathfrak X (M)$ and $Y_1, Y_2 \in \mathfrak X (N)$ are vector fields such that $X_i$ and $Y_i$ are $F$-related, $i = 1, 2$, then the vector fields $\nabla^M_{X_1} X_2$ and $\nabla^N _{Y_1}Y_2$ are also $F$-related. There are several equivalent ways to express the $F$-relatedness of two connections, according to the following

\begin{lemma}\label{lem:conn,F-rel}
Let $\nabla^M$ and $\nabla^N$ be connections in $TM$ and $TN$ respectively, and let $F : M \to N$ be a smooth map. The following conditions are equivalent:
\begin{enumerate}
\item $\nabla^M$ and $\nabla^N$ are $F$-related;
\item $\nabla^M (F^\ast \theta) = F^\ast (\nabla^N \theta)$, for all $\theta \in \Omega^1 (M)$;
\item both the symmetric parts $(\nabla^M)^{\mathrm{sym}}, (\nabla^N)^{\mathrm{sym}}$ and the torsions $T^{\nabla^M}, T^{\nabla^N}$ of $\nabla^M, \nabla^N$ are $F$-related;
\item the tangent map $dF : TM \to TN$ maps the horizontal distribution $H^{\nabla^M}$ to $H^{\nabla^N}$;
\item the vertical projections $\theta^{\nabla^M}, \theta^{\nabla^N}$ are $dF$-related;
\item the torsions $T^{\nabla^M}, T^{\nabla^N}$ are $F$-related and the geodesic sprays $Z^{\nabla^M}, Z^{\nabla^N}$ are $dF$-related.
\end{enumerate}
\end{lemma}

We conclude this introductory section discussing the product of two manifolds equipped with linear connections in their tangent bundles. Given two manifolds $M, N$ we will usually denote by $\mathrm{pr}_M : M \times N \to M$ and $\mathrm{pr}_N : M \times N \to N$ the projections of the product $M \times N$ onto its factors. It is clear that there exists a unique connection $\nabla^M \times \nabla^N$ in the tangent bundle $T(M \times N)$ of the product $M \times N$, such that $\nabla^M \times \nabla^N, \nabla^M$ are $\mathrm{pr}_M$-related, and $\nabla^M \times \nabla^N, \nabla^N$ are $\mathrm{pr}_N$-related.
We will often need the following straightforward lemma that we state for future reference.

\begin{lemma}\label{lem:fibered}
Let $\sigma : M \to B$, and $\tau : N \to B$ be surjective submersions, and let $\nabla^M, \nabla^N$ be connections in $TM, TN$ respectively. Suppose that there exists a connection $\nabla^B$ in $TB$ such that $\nabla^M, \nabla^B$ are $\sigma$-related and $\nabla^N, \nabla^B$ are $\tau$-related. Then the product connection $\nabla^M \times \nabla^N$in $T(M \times N)$ restricts to the fiber product $M \times_B N \subseteq M \times N$, i.e. there exists a (necessarily unique) connection $\nabla^{\mathrm{fib}}$ in $T(M \times_B N)$ such that $\nabla^{\mathrm{fib}}, \nabla^M \times \nabla^N$ are $j$-related, where $j : M \times_B N \hookrightarrow M \times N $ is the inclusion.
\end{lemma}

\section{Multiplicative Connections}\label{Sec:Mult_Conn}

In this section we introduce a notion of multiplicative connection in the tangent bundle of a Lie groupoid and discuss its main properties. 

We begin recalling the notion of multiplicative tensor from \cite{BD2019}. We will only need $(p,1)$-tensors for which we also refer to \cite{BD2013} (see also \cite{LSX2009}). So let $\mathcal G \rightrightarrows M$ be a Lie groupoid. We will denote by $s,t,u,m,i$ the structure maps in $\mathcal G$ (source, target, unit, multiplication, inversion, respectively). The unit $u(x)$ at a point $x \in M$ will be also denoted $1_x$, the multiplication $m(g,g')$ of two composable arrows $g,g' \in \mathcal G$ will be also denoted $gg'$, and the inverse $i(g)$ of an arrow $g$ will be also denoted $g^{-1}$. The manifold of composable arrows will be denoted
\[
\mathcal G^{(2)} =\left\{(g,g') \in \mathcal G \times \mathcal G : s(g) = t(g') \right\}.
\] 
We denote by $j_{\mathcal G^{(2)}} : \mathcal G^{(2)} \hookrightarrow \mathcal G \times \mathcal G$ the inclusion.

There are various equivalent definitions of multiplicative tensor on $\mathcal G$. For $(p,1)$-tensors we will adopt the one provided by a characterization in \cite{BD2013}:
let $\mathcal T$ be a $(p,1)$-tensor on $\mathcal G$. We say that $\mathcal T$ is \emph{$(s,t)$-projectable} if there exists a, necessarily unique, $(p,1)$-tensor $\mathcal T^M$ on $M$ such that $\mathcal T, \mathcal T^M$ are both $s$-related and $t$-related. In this case, we call $\mathcal T^M$ the $(s,t)$-projection of $\mathcal T$. It is easy to see that if $\mathcal T$ is $(s,t)$-projectable, then $\mathcal T \times \mathcal T$ restricts to $\mathcal G^{(2)}$, i.e. there also exists a, necessarily unique, $(p,1)$-tensor $\mathcal T^{(2)}$ on $\mathcal G^{(2)}$ such that $\mathcal T^{(2)}, \mathcal T \times \mathcal T$ are $j_{\mathcal G^{(2)}}$-related. Here $\mathcal T \times \mathcal T$ is the unique $(p,1)$-tensor on $\mathcal G \times \mathcal G$ such that $\mathcal T \times \mathcal T, \mathcal T$ are both $\mathrm{pr}_1$-related and $\mathrm{pr}_2$-related, and $\mathrm{pr}_i : \mathcal G\times \mathcal G \to \mathcal G$ is the projection onto the $i$-th factor, $i = 1,2$. Finally, a $(p,1)$-tensor $\mathcal T$ on $\mathcal G$ is \emph{multiplicative} if it is $(s,t)$-projectable, and $\mathcal T^{(2)}, \mathcal T$ are $m$-related (we stress again that this is not the definition in \cite{BD2019}, nor that in \cite{LSX2009}, but it is an equivalent one, and the equivalence is proved in \cite{BD2013}, at least for skew-symmetric $(p,1)$-tensors. Notice that the same proof is still valid without the skew-symmetry assumption).

This definition 
translates to a definition of \emph{multiplicative connection}. Namely, let $\nabla$ be a connection in the tangent bundle $ T\mathcal G $ of a Lie groupoid $\mathcal G$.

\begin{definition}
The connection $\nabla$ is $(s,t)$\emph{-projectable} if there exists a, necessarily unique, connection $\nabla^M$ in $TM$, such that $\nabla, \nabla^M$ are both $s$-related and $t$-related. In this case we call $\nabla^M$ the $(s,t)$\emph{-projection} of $\nabla$.
\end{definition}

Notice that, in view of Lemma \ref{lem:fibered}, if the connection $\nabla$ in $ T\mathcal G $ is $(s,t)$-projectable, then the connection $\nabla^\times = \nabla \times \nabla$ in the tangent bundle of $\mathcal G \times \mathcal G$ restricts to $\mathcal G^{(2)}$, i.e. there exists a, necessarily unique, connection $\nabla^{(2)}$ in the tangent bundle of $\mathcal G^{(2)}$ such that $\nabla^{(2)}, \nabla^\times$ are $j_{\mathcal G^{(2)}}$-related.

\begin{definition}
A connection $\nabla$ in the tangent bundle $ T\mathcal G $ of a Lie groupoid $\mathcal G \rightrightarrows M$ is \emph{multiplicative} if it is $(s,t)$-projectable, and $\nabla^{(2)}, \nabla$ are $m$-related.
\end{definition}

Multiplicative connections can be characterized in various ways according to the following

\begin{proposition}\label{prop:char_mult_conn}
Let $\nabla$ be a connection in the tangent bundle $ T\mathcal G $ of a Lie groupoid $\mathcal G \rightrightarrows M$. The following conditions are equivalent:
\begin{enumerate}
\item $\nabla$ is a multiplicative connection;
\item the symmetric part $\nabla^{\mathrm{sym}}$ and the torsion $T^\nabla$ of $\nabla$ are multiplicative;
\item the horizontal distribution $H^\nabla$ is a multiplicative distribution on $ T\mathcal G  \rightrightarrows TM$ (i.e. it is a Lie subgroupoid in the Lie groupoid $T T\mathcal G  \rightrightarrows TTM$);
\item the vertical projection $\theta^\nabla$ is a multiplicative $(1,1)$-tensor on $ T\mathcal G  \rightrightarrows TM$;
\item the torsion $T^\nabla$ is multiplicative, and the geodesic spray $Z^\nabla$ is a multiplicative vector field on $ T\mathcal G  \rightrightarrows TM$.
\end{enumerate}
\end{proposition}

\begin{proof}
The statement is a straightforward consequence of Lemma \ref{lem:conn,F-rel} and we leave the simple details to the reader. We only remark here that the equivalence $(3) \Leftrightarrow (4)$ is discussed in \cite{BD2019}.
\end{proof}

\begin{remark}
Let $\mathcal G \rightrightarrows M$ be a Lie groupoid. It is clear that the space of multiplicative connections in the tangent bundle $ T\mathcal G $ is either empty or it is an affine space modeled on the vector space of multiplicative $(2,1)$-tensors on $\mathcal G$. 
\end{remark}

We postpone examples to Section \ref{Sec:Examples}. Here we study various useful properties of multiplicative connections.

\begin{proposition}\label{prop:simpl_conn}
Let $\mathcal G \rightrightarrows M$ be a Lie groupoid, and let $\nabla$ be a multiplicative connection in $ T\mathcal G $. Denote by $\nabla^M$ (resp., $\nabla^{(2)}$) the $(s,t)$-projection of $\nabla$ (resp., the restriction of $\nabla \times \nabla$ to $\mathcal G^{(2)}$). Consider also the embeddings
\[
\epsilon_1 : \mathcal G \to \mathcal G^{(2)}, \quad g \mapsto (1_{t(g)}, g), \quad \text{and} \quad \epsilon_2 : \mathcal G \to \mathcal G^{(2)}, \quad g \mapsto (g, 1_{s(g)}).
\]
Then 
\begin{enumerate}
\item $\nabla^M, \nabla$ are $u$-related,
\item $\nabla, \nabla^{(2)}$ are $\epsilon_i$-related, $i = 1, 2$, and
\item $\nabla$ is self-$i$-related.
\end{enumerate}
\end{proposition}

\begin{proof}
The simplest proof is using that the analogous statement holds for $(p,1)$-tensors (see \cite[Section 4]{BD2013}) together with Lemma \ref{lem:conn,F-rel}.(6) and the simple fact that the torsion and the geodesic spray of $\nabla \times \nabla$ are $T^\nabla \times T^\nabla$ and $Z^\nabla \times Z^\nabla$ respectively. We leave the straightforward details to the reader. 
\end{proof}

\begin{remark}
Let $\mathcal G \rightrightarrows M$ be a Lie groupoid, and let $\nabla$ be a multiplicative connection in $ T\mathcal G $. Recall that the \emph{nerve} of $\mathcal G$ is the simplicial manifolds $\mathcal G^{(\bullet)}$ whose $k$-th level is the space $\mathcal G^{(k)}$ of strings of $k$ composable arrows in $\mathcal G$, together with the usual faces and degeneracies as structure maps. It is easy to see along similar lines as above that, for all $k$, the product connection
\[
\underset{\text{$k$ times}}{\underbrace{\nabla \times \dots \times \nabla}}
\]
restricts to a unique connection $\nabla^{(k)}$ in the space $\mathcal G^{(k)}$ of composable arrows. Additionally the $\nabla^{(k)}$'s are all related via the structure maps of the nerve. This can be rephrased saying that a multiplicative connection in $ T\mathcal G $ determines in a canonical way a \emph{simplicial connection} in the simplicial vector bundle $ T\mathcal G ^{(\bullet)} \to \mathcal G^{(\bullet)}$.
\end{remark}

We now discuss existence of multiplicative connections.
As we will see in the next section, not all Lie groupoids admit multiplicative connections in their tangent bundles. Here, we identify the \emph{obstruction} to the existence. It is a canonical cohomology class in a certain cochain complex naturally attached to every Lie groupoid. As we will briefly discuss, such complex is a close relative of Crainic-Mestre-Struchiner \emph{deformation complex} of the Lie groupoid \cite{CMS2018}. We believe that our ``\emph{obstruction class}'' can be seen as a groupoid (global) version of the \emph{Atiyah class of a dg-vector bundle} discussed in \cite{MSX2015}. See Appendix \ref{app:A} for a short comment on this.

We begin our discussion characterizing multiplicative connections in the most appropriate way for our purposes. To do this it is helpful to introduce the \emph{division map}:
\[
\bar m : \mathcal G^{[2]} := \mathcal G \mathbin{{}_s \times_s} \mathcal G \to \mathcal G, \quad (g,h) \mapsto gh{}^{-1}.
\]
If $\nabla$ is an $s$-projectable connection in $T\mathcal G$ then, from Lemma \ref{lem:fibered}, the connection $\nabla \times \nabla$ restricts to a unique connection $\nabla^{[2]}$ in $T \mathcal G^{[2]}$.

\begin{proposition}\label{prop:mult_conn_div}
Let $\mathcal G \rightrightarrows M$ be a Lie groupoid, and let $\nabla$ be a linear connection in $ T\mathcal G $. The following two conditions are equivalent:
\begin{enumerate}
\item $\nabla$ is a multiplicative connection;
\item $\nabla$ is $s$-projectable and $\nabla^{[2]}, \nabla$ are $\bar m$-related.
\end{enumerate}
\end{proposition}

\begin{proof}
There is a simple proof in the same spirit as that of Proposition \ref{prop:simpl_conn} and we omit it.
\end{proof}

From now on, given a manifold $N$, we will denote by $T^{p,q} N = (T^\ast
N)^{\otimes p} \otimes (T N)^{\otimes q}$
the bundle of
$(p, q)$-tensors on $N$, and by $\mathfrak T^{p,q}(N) = \Gamma(T^{p,q}N)$ the module of sections of $T^{p,q}N$. Now, take any $s$-projectable connection $\nabla$ in $ T\mathcal G  \to \mathcal G$. Such connection always exists. Denote by $\nabla^M$ its $s$-projection. The combination
\[
 \bar m{}^\ast \circ \nabla - \nabla^{[2]} \circ \bar m{}^\ast : \Omega^1 (\mathcal G) \to \mathfrak T^{2,0}(\mathcal G^{[2]})
\]
is a module homomorphism covering the algebra homomorphism $\bar m{}^\ast : C^\infty (\mathcal G) \to C^\infty (\mathcal G^{[2]})$, hence, it can be seen as a vector bundle map $\bar m{}^\ast T^\ast \mathcal G \to T^{2,0} \mathcal G^{[2]}$, or, equivalently, as a vector bundle map
\[
\mathsf{At}_\nabla : T^{0,2} \mathcal G^{[2]} \to \bar m{}^\ast  T\mathcal G .
\]
The notation $\mathsf{At}$ stands for ``Atiyah'' and it is explained in Appendix \ref{app:A}. In the same way, the combination
\[
t^\ast \circ \nabla^M - \nabla \circ t^\ast : \Omega^1 (M) \to \mathfrak T^{2,0}(\mathcal G)
\]
can be seen as a vector bundle map
\[
\mathcal T^M_\nabla : T^{0,2} \mathcal G \to t^\ast TM.
\]
It is easy to see that the following diagram commutes
\[
\begin{array}{c}
\xymatrix{ T^{0,2} \mathcal G^{[2]} \ar[r]^-{\mathsf{At}_\nabla} \ar[d]_-{(d \mathrm{pr}_2){}^{\otimes 2}} & \bar m{}^\ast  T\mathcal G  = \mathcal G^{[2]} \mathbin{{}_{\bar m} \times}  T\mathcal G  \ar[d]^-{\mathrm{pr}_2 \times ds} \\
T^{0,2} \mathcal G \ar[r]^-{\mathcal T^M_\nabla} & t^\ast TM = \mathcal G \mathbin{{}_t \times} TM}
\end{array}.
\]

Actually, $\mathsf{At}_\nabla$ is a $1$-cocycle in an appropriate cochain complex $\bar C{}^\bullet_{\mathrm{def}} (\mathcal G, T^{2,0})$ that we now define. First of all, $\bar C{}^\bullet_{\mathrm{def}} (\mathcal G, T^{2,0})$ will be concentrated in degrees $k \geq -1$. Degree $-1$ cochains are vector bundle maps
\[
U : T^{0,2} M \to A
\]
where $A = \mathrm{Lie}(\mathcal G)$ is the Lie algebroid of $\mathcal G$. Degree $0$ cochains are $s$-projectable $(2,1)$-tensors on $\mathcal G$, i.e. $(2,1)$-tensors $U \in \mathfrak T^{2,1}(\mathcal G)$ for which there exists a $(2,1)$-tensor $U^M \in \mathfrak T^{2,1} (M)$ such that the following diagram commutes:
\[
\begin{array}{c}
\xymatrix{ T^{0,2} \mathcal G \ar[r]^-{U} \ar[d]_-{ds^{\otimes 2}} &  T\mathcal G  \ar[d]^-{ds} \\
T^{0,2} M \ar[r]^-{U^M} & TM}
\end{array}.
\]
In order to describe $k$-cochains for $k > 0$, we introduce the manifold
\[
\mathcal G^{[k+1]} = \underset{\text{$k$ times}}{\underbrace{\mathcal G \times_s  \cdots \times_s \mathcal G}}
\]
of $(k+1)$-tuples of arrows in $\mathcal G$ with the same source. As
\[
 T\mathcal G ^{[k+1]} \cong \underset{\text{$k$ times}}{\underbrace{ T\mathcal G  \times_{ds} \cdots \times_{ds}  T\mathcal G }},
\]
in the following, we will always interpret a tangent vector to $\mathcal G^{[k+1]}$ as a $(k+1)$-tuple of tangent vectors to $\mathcal G$ (with the same $ds$-projection). We also introduce a special notation for a decomposable element in $T^{0,2} \mathcal G^{[k+1]}$. Namely, let $(v_1^1, \ldots, v_{k+1}^1), (v_1^2, \ldots, v_{k+1}^2) \in T \mathcal G^{[k+1]}$ be tangent vectors at the same point of $\mathcal G^{[k+1]}$. We denote
\[
(v_1^\otimes, \ldots, v_{k+1}^\otimes) := (v_1^1, \ldots, v_{k+1}^1) \otimes (v_1^2, \ldots, v_{k+1}^2) \in T^{0,2} \mathcal G^{[k+1]}. 
\]

The space $\bar C{}^k_{\mathrm{def}} (\mathcal G, T^{2,0})$ now consists of smooth maps
\[
U : T^{0,2} \mathcal G^{[k+1]} \to  T\mathcal G , 
\]
such that 
\begin{enumerate}
\item $U$ is a vector bundle map 
\[
U : T^{0,2} \mathcal G^{[k+1]} \to (\bar m \circ (\mathrm{pr}_1 \times \mathrm{pr}_2))^\ast  T\mathcal G 
\]
where $\mathrm{pr}_i : \mathcal G^{[k+1]} \to \mathcal G$ is the projection onto the $i$-th entry, and
\item there exists another vector bundle map
\[
U^M : T^{0,2} \mathcal G^{[k]} \to (t \circ \mathrm{pr}_1)^\ast TM
\]
such that the following diagram commutes: 
\[
\begin{array}{c}
\xymatrix{ T^{0,2} \mathcal G^{[k+1]} \ar[r]^-{U} \ar[d]_-{d(\mathrm{pr}_2 \times \cdots \times \mathrm{pr}_{k+1})^{\otimes 2}} & (\bar m \circ (\mathrm{pr}_1 \times \mathrm{pr}_2))^\ast  T\mathcal G  \ar[d]^-{(\mathrm{pr}_2 \times \cdots \times \mathrm{pr}_{k+1}) \times ds} \\
T^{0,2} \mathcal G^{[k]} \ar[r]^-{U^M} & (t \circ \mathrm{pr}_1)^\ast TM}
\end{array}.
\]
\end{enumerate}

For instance, given any $s$-projectable connection $\nabla$, the vector bundle map $\mathsf{At}_\nabla$ defined above is a degree $1$-cochain in $\bar C{}^\bullet_{\mathrm{def}} (\mathcal G, T^{2,0})$. For every $U \in \bar C{}^{\geq 0}_{\mathrm{def}} (\mathcal G, T^{2,0})$, we call $U^M$ the $M$-projection of $U$.

Next we describe the differential $\bar \delta : \bar C{}^\bullet_{\mathrm{def}} (\mathcal G, T^{2,0}) \to \bar C{}^{\bullet + 1}_{\mathrm{def}} (\mathcal G, T^{2,0})$.  We begin with $\bar \delta : \bar C{}^{-1}_{\mathrm{def}} (\mathcal G, T^{2,0}) \to \bar C{}^{0}_{\mathrm{def}} (\mathcal G, T^{2,0})$. So take a $-1$-cochain $U : T^{0,2} M \to A$. For every $g \in \mathcal G$, and $v,v' \in T_g G$, we put
\[
\bar \delta U (v \otimes v') = (dR_g \circ U \circ dt^{\otimes 2}) (v \otimes v') + (dL_g \circ di \circ U \circ ds{}^{\otimes 2} )(v \otimes v'),
\]
where $R_g$ (resp., $L_g$) denotes right (resp., left) translation by $g$. It is easy to see that $\bar \delta U$ is a well defined $0$-cochain, whose $M$-projection $(\bar \delta U)^M$ is given by
\[
(\bar \delta U)^M = \rho_A \circ U : T^{2,0} M \to TM,
\]
where $\rho_A : A \to TM$ is the anchor map.

Now let $U \in  \bar C{}^{k}_{\mathrm{def}} (\mathcal G, T^{2,0})$, with $k \geq 0$. For all $(v^\otimes_1, \ldots, v^\otimes_{k+2}) \in T^{0,2} \mathcal G^{[k+2]}$ we put
\begin{equation}\label{eq:delta_bar}
\begin{aligned}
 \bar \delta U (v^\otimes_1, \ldots, v^\otimes_{k+2}) 
& = - d \bar m \big( U \left( v^\otimes_1, v^\otimes_3, \ldots, v^\otimes_{k+2}\right), U \left( v^\otimes_2, \ldots, v^\otimes_{k+2}\right) \big) \\
& \quad + \sum_{i = 3}^{k+2} (-)^{i+1} U\left(v^\otimes_1, \ldots, \widehat{v^\otimes_i}, \ldots, v^\otimes_{k+2} \right) \\
& \quad + (-)^{k} U \big( d \bar m{}^{\otimes 2} (v^\otimes_1, v^\otimes_{k+2}), \ldots, d \bar m{}^{\otimes 2} (v^\otimes_{k+1}, v^\otimes_{k+2}) \big),
\end{aligned}
\end{equation}
where a hat ``$\widehat{-}$'' denotes omission. A straightforward computation exploiting that $s \circ \bar m = t \circ \mathrm{pr}_2$ shows that $\bar \delta U$ is a well-defined $(k+1)$-cochain whose $M$-projection $(\bar \delta U)^M$ is given by
\[
\begin{aligned}
 (\bar \delta U)^M (v^\otimes_2, \ldots, v^\otimes_{k+2})
& = -dt \big(U (v^\otimes_2, \ldots, v^\otimes_{k+2}) \big) \\
& \quad + \sum_{i = 3}^{k+2} (-)^{i+1} U^M (v^\otimes_2, \ldots, \widehat{v^\otimes_i}, \ldots, v^\otimes_{k+2}) \\
& \quad + (-)^{k} U^M \big(d \bar m{}^{\otimes 2} (v^\otimes_2, v^\otimes_{k+2}), \ldots, d \bar m{}^{\otimes 2} (v^\otimes_{k+1}, v^\otimes_{k+2})\big).
\end{aligned}
\]
A direct check reveals that $\bar \delta$ is indeed a differential.

\begin{remark}
The definition of $\bar C{}^{\bullet}_{\mathrm{def}} (\mathcal G, T^{2,0})$ is very similar to that of Crainic-Mestre-Struchiner deformation complex of a Lie groupoid $\mathcal G \rightrightarrows M$ in the $\bar m$-version $\bar C{}^{\bullet}_{\mathrm{def}}(\mathcal G)$ (see \cite[Appendix]{CMS2018}), and the definition of the differential $\bar \delta$ is formally identical up to replacing points in $\mathcal G^{[k]}$ with points in $T^{2,0} \mathcal G^{[k]}$. This is the main reason why we chose a similar notation $\bar C{}^{\bullet}_{\mathrm{def}} (\mathcal G, T^{2,0})$ for our complex. Notice that the proof that $\bar \delta$ is a differential can be performed exactly as in \cite{CMS2018}. We leave details to the reader.
\end{remark}

\begin{remark}\label{rem:delta_2,1_tensors}
$0$-cocycles in the deformation complex $\bar C{}^{\bullet}_{\mathrm{def}}(\mathcal G)$ of Crainic-Mestre-Struchiner are multiplicative vector fields on $\mathcal G$. Similarly, $0$-cocycles in the complex $\bar C{}^{\bullet}_{\mathrm{def}}(\mathcal G, T^{2,0})$ are exactly multiplicative $(2,1)$-tensors. We leave details to the reader and we only sketch the main steps of the proof. Let $\mathcal T \in \mathfrak T^{2,1} (\mathcal G)$ and interpret it as a module homomorphisms
$
\mathcal T : \Omega^1 (\mathcal G) \to \mathfrak T^{2,0}(\mathcal G)
$.
If $\mathcal T \in \mathfrak T^{2,1} (\mathcal G)$ is $s$-projectable, then the tensor $\mathcal T \times \mathcal T$ restricts to a unique tensor $\mathcal T^{[2]} \in \mathfrak T^{2,1} (\mathcal G^{[2]})$. Similarly as for connections in Proposition \ref{prop:mult_conn_div}, $\mathcal T$ is a multiplicative tensor iff the combination
$
\mathcal S = \bar m{}^{\ast } \circ \mathcal T - \mathcal T^{[2]} \circ \bar m{}^\ast
$
vanishes. Now, $\mathcal T$ can be interpreted as a $0$-cochain in $\bar C{}^{\bullet}_{\mathrm{def}}(\mathcal G, T^{2,0})$. Similarly, $\mathcal S$ can be interpreted as a $1$-cochain in $\bar C{}^{\bullet}_{\mathrm{def}}(\mathcal G, T^{2,0})$ and, if we do so, a straightforward computation reveals that $\mathcal S = \bar \delta \mathcal T$.

Finally, there are obvious versions of the cochain complex $\bar C{}^{\bullet}_{\mathrm{def}}(\mathcal G, T^{2,0})$ where the tensor bundles $T^{2,0}\mathcal G^{[k]}$ are replaced by $T^{p,0}\mathcal G^{[k]}$, for arbitrary $p$, and such that $1$-cocycles are exactly multiplicative $(p,1)$-tensors.
\end{remark}

The following theorem should be compared with \cite[Proposition 2.1]{MSX2015}.

\begin{theorem}\label{theor:obst_class}
Let $\mathcal G \rightrightarrows M$ be a Lie groupoid and pick any $s$-projectable connection $\nabla$ in $ T\mathcal G $. Consider the $1$-cochain $\mathsf{At}_\nabla \in \bar C{}^1_{\mathrm{def}}(\mathcal G, T^{2,0})$. Then
\begin{enumerate}
\item $\mathsf{At}_\nabla$ is a $1$-cocycle, i.e. $\bar \delta \mathsf{At}_\nabla = 0$;
\item the cohomology class $[\mathsf{At}_\nabla] \in H^1 (\bar C{}^\bullet_{\mathrm{def}}(\mathcal G, T^{2,0}), \bar \delta)$ is independent of the choice of $\nabla$;
\item the cohomology class $[\mathsf{At}_\nabla]$ vanishes iff there exists a multiplicative connection in $ T\mathcal G $.
\end{enumerate} 
\end{theorem}

\begin{proof}
We couldn't find any simple proof in terms of $T^\nabla, Z^\nabla$. We propose a direct proof in Appendix \ref{app:B}.
\end{proof}

\section{IM Connections and the Lie Theory of Multiplicative Connections}\label{Sec:IM_Conn}

In this section we introduce the infinitesimal analogues of multiplicative connections, that we call \emph{infinitesimally multiplicative} (IM) connections. Before giving the definition, it is useful to review from (\cite{BD2019}) the definition of IM tensor. We warn the reader that our presentation slightly deviates from Bursztyn-Drummond one. 
We will only need skew-symmetric $(p, 1)$-tensors, i.e. vector valued $p$-forms.

Let $A \Rightarrow M$ be a Lie algebroid. 

\begin{definition}[{\cite{BD2019}, see also \cite{BD2013}}]\label{def:IM_tensor}
An \emph{IM vector valued $p$-form} on $A$  is a triple $(\mathcal D, l, \mathcal T^M)$
consisting of
\begin{enumerate}
\item a first order differential operator $\mathcal D : \Gamma (A) \to \Omega^p (M, A)$,
\item a vector bundle map $l : A \to \wedge^{p-1}T^\ast M \otimes A$,
\item a vector valued $p$-form $\mathcal T^M \in \Omega^p (M, TM)$,
\end{enumerate}
such that
\begin{equation}\label{eq:calD}
\mathcal D(fa) = f \mathcal D(a) + df \wedge l(a) - \left\langle df, \mathcal T^M \right\rangle \otimes a, \quad f \in C^\infty (M), \quad a \in \Gamma (A),
\end{equation}
and additionally satisfying the following \emph{IM equations}:
\begin{align}
\mathcal D([a,b]_A) & = \mathcal L_a \mathcal D(b) - \mathcal L_b \mathcal D(a), \label{eq:IM_tensor_1}\\
l ([a,b]_A) & = \mathcal L_a l(b) - \iota_{\rho_A(b)} \mathcal D(a), \label{eq:IM_tensor_2}\\
 \mathcal L_{\rho_A(a)} \mathcal T^M& = \rho_A \circ \mathcal D(a), \label{eq:IM_tensor_3}\\
\iota_{\rho_A(a)}l(b) & = - \iota_{\rho_A(b)}l(a), \label{eq:IM_tensor_4}\\
\iota_{\rho_A (a)} \mathcal T^M& = \rho_A \circ  l(a) \label{eq:IM_tensor_5}
\end{align}
for all $a, b \in \Gamma (A)$. 
\end{definition}

Formulas (\ref{eq:calD}) and the IM equations (\ref{eq:IM_tensor_1})--(\ref{eq:IM_tensor_5}) require some explanations. First of all, the bracket $\langle -, -\rangle$ in the last term of the rhs of (\ref{eq:calD}) denotes the contraction of $df$ with the vector component of $\mathcal T^M$. Next by $\mathcal L_a \mathcal D(b)$ (and likewise $\mathcal L_b \mathcal D(a)$ and $\mathcal L_a l(b)$) in (\ref{eq:IM_tensor_1}) we mean the \emph{Lie derivative of the $A$-valued $p$-form $\mathcal D(b)$ along the derivation $[a,-]_A$ of $A$}, i.e. the $A$-valued $p$-form given by
\begin{equation}\label{eq:Lie_aD(b)}
\begin{aligned}
& \mathcal L_a \mathcal D(b) (X_1, \ldots, X_p) \\
& = \left[a, \mathcal D(b)(X_1, \ldots, X_p)\right]_A - \sum_{i = 1}^p \mathcal D(b) \left(X_1, \ldots, [\rho_A(a), X_i], \ldots, X_p \right), \quad X_i \in \mathfrak X (M).
\end{aligned}
\end{equation}

\begin{remark}\label{rem:relations}
Formulas (\ref{eq:IM_tensor_1})--(\ref{eq:IM_tensor_5}) specialize to the case of vector valued forms the Formulas (IM1)--(IM4), (IM6) in \cite[Theorem 1.1]{BD2019}, respectively. Formula (IM5) therein is empty in our case. Not all the identities (\ref{eq:IM_tensor_1})--(\ref{eq:IM_tensor_5}) are independent. For instance (\ref{eq:IM_tensor_1}) and (\ref{eq:IM_tensor_2}) together imply (\ref{eq:IM_tensor_3}) (just use (\ref{eq:IM_tensor_1}) and (\ref{eq:IM_tensor_2}) to compute $\mathcal D([a,fb])$, with $f \in C^\infty (M)$, to check it). In its turn (\ref{eq:IM_tensor_3}) implies (\ref{eq:IM_tensor_5}) (compute $\langle \mathcal L_{\rho_A(fa)} \mathcal T, \theta \rangle$).
\end{remark}

The definition of IM vector valued form is motivated by the following 

\begin{theorem}[\cite{BD2019}]\label{theor:Lie_mult_forms}
Let $\mathcal G \rightrightarrows M$ be a Lie groupoid, let $A \Rightarrow M$ be its Lie algebroid, and let $\mathcal T$ be a multiplicative vector valued $p$-form on $\mathcal G$. Then $\mathcal T$ determines an IM vector valued $p$-form $(\mathcal D, l, \mathcal T^M)$ on $M$ via
\begin{equation*}
\begin{aligned}
\mathcal D(a) & = P(\mathcal L_{\overrightarrow a} \mathcal T), \\
l(a) & = P(\iota_{\overrightarrow a} \mathcal T), \\
\mathcal T^M (v_1, \ldots, v_p) & = ds \left( \mathcal T(v_1, \ldots, v_p ) \right),
\end{aligned}
\end{equation*}
for all $a \in \Gamma (A)$, and all $v_1, \ldots, v_p \in TM$, where $\overrightarrow a$ is the right invariant vector field on $\mathcal G$ corresponding to $a$, and $P : \Omega^p(\mathcal G, T\mathcal G) \to \Omega^p(M, A)$ is the projection defined by $P(\mathcal S) = \operatorname{pr}_A \circ \, u^\ast \mathcal S$, (with $\operatorname{pr}_A : T\mathcal G|_M \to A$ the canonical projection with kernel $TM$). If $\mathcal G$ is source $1$-connected, then the assignment $\mathcal T \mapsto (\mathcal D, l, \mathcal T^M)$ is a bijection. 
\end{theorem}

\begin{remark}
Let $A \to M$ be a vector bundle, not necessarily a Lie algebroid. The vector bundle $A \to M$ can be seen as the source $1$-connected Lie groupoid with source and target being both the projection $A \to M$, and multiplication being the fiber addition. The Lie algebroid of this groupoid is $A$ itself with trivial bracket and trivial anchor. A \emph{linear} vector valued form on the vector bundle $A$ is a multiplicative vector valued form $\mathcal T \in \Omega^\bullet (A, TA)$ wrt the latter Lie groupoid structure \cite{BD2019}. The term \emph{linear} emphasizes that $\mathcal T$ is compatible with the vector bundle structure. Finally, when specialized to this case, Theorem \ref{theor:Lie_mult_forms} establishes a bijection between linear vector valued forms on $A$ and triples $(\mathcal D, l, \mathcal T^M)$ as in Definition \ref{def:IM_tensor} satisfying (\ref{eq:calD}), but not-necessarily satisfying (\ref{eq:IM_tensor_1})--(\ref{eq:IM_tensor_5}). For this reason, the latter triples will be called the \emph{components of a linear vector valued form}. We stress that it is not really necessary to use Theorem \ref{theor:Lie_mult_forms} to show that linear vector valued forms correspond bijectively to their components, and the latter bijection (that can be proved, e.g., in local coordinates) is actually a step towards the proof of Theorem \ref{theor:Lie_mult_forms} (see \cite{BD2019}).
\end{remark}

We now turn to connections. Let $A \Rightarrow M$ be a Lie algebroid again.

\begin{definition}\label{def:IM_conn}
An \emph{IM connection} in the tangent bundle $TA \to A$ of $A$ is a $4$-tuple $(\mathcal F, \nabla^A, \nabla^M, l)$ consisting of \begin{enumerate}
\item a second order differential operator $\mathcal F : \Gamma (A) \to \Gamma (T^{2,0} \otimes A)$,
\item a linear connection $\nabla^A$ in $A$,
\item a linear connection $\nabla^M$ in $M$,
\item a vector bundle map $l : E \to T^\ast M \otimes A$,
\end{enumerate}
such that 
\begin{equation}\label{eq:calF}
\mathcal F (fa) = f \mathcal F (a) + \nabla^M df \otimes a + df \odot \nabla^A a + df \otimes l (a), \quad f \in C^\infty (M), \quad a \in \Gamma (A),
\end{equation}
 and additionally satisfying the following identities
\begin{align}
[a,b]_A & = \nabla^A_{\rho_A(a)}b - \nabla^A_{\rho_A(b)}a- \iota_{\rho_A(b)}l(a) \label{eq:IM_conn_1}\\
\mathcal L_a \nabla^A (X, b) & = \mathcal F (a) ( X , \rho_A(b)) \label{eq:IM_conn_2} \\
\mathcal F ([a,b]_A) & = \mathcal L_a \mathcal F (b) - \mathcal L_b \mathcal F (a) \label{eq:IM_conn_3} \\
l ([a,b]_A) & = \mathcal L_a l(b) - \iota_{\rho_A(b)} \mathcal D(a), \label{eq:IM_conn_4}
\end{align}
for all $a,b \in \Gamma (A)$, $X \in \mathfrak X (M)$. The \emph{IM torsion} of $(\mathcal F, \nabla^A, \nabla^M, l)$ is the triple $(\mathcal D, l, T^M)$, where $T^M$ is the torsion of $\nabla^M$, and $\mathcal D : \Gamma (A) \to \Omega^2 (M, A)$ is given by
\begin{equation}\label{eq:IM_torsion}
\mathcal D (a)(X_1, X_2) = \mathcal F(a) (X_1, X_2) - \mathcal F (a) (X_2, X_1), \quad X_1, X_2 \in \mathfrak X (M).
\end{equation}
\end{definition}

Some explanations are necessary. In (\ref{eq:calF}) $df \odot \nabla^A a \in \Gamma (T^{2,0} M \otimes A)$ is given by
\[
df \odot \nabla^A a \ (X_1, X_2) = X_1(f) \nabla^A_{X_2} a + X_2(f) \nabla^A_{X_1} a, \quad X_1, X_2 \in \mathfrak X (M).
\]
Moreover, the Lie derivative $\mathcal L_a \mathcal F(b)$ of the $A$-valued $(2,0)$-tensor $\mathcal F(b)$ is given by the same formula (\ref{eq:Lie_aD(b)}) as for $\mathcal D (b)$ (but with $p = 2$), and the Lie derivative $\mathcal L_a \nabla^A$ of the connection $\nabla^A$ is the $\operatorname{End} A$-valued $1$-form given by the similar formula:
\[
\mathcal L_a \nabla^A (X, b) = \left[ a, \nabla^A_X b \right]_A - \nabla^A_{[\rho_A(a), X]} b - \nabla^A_X [a, b]_A, \quad X \in \mathfrak X (M), \quad b \in \Gamma (A).
\]

Definition \ref{def:IM_conn} is motivated by Theorems \ref{theor:IM_spray} and \ref{theor:Lie_theory} below.

\begin{remark}
Various identities follow from (\ref{eq:IM_conn_1})--(\ref{eq:IM_conn_3}). Namely, from (\ref{eq:IM_conn_1}) we have
\begin{equation}\label{stanco}
\iota_{\rho_A (a)}l(b) = - \iota_{\rho_A(b)}l (a).
\end{equation}
Additionally, computing $\mathcal L_{fa} \nabla^A (X, b)$ for all $f \in C^\infty (M)$ via (\ref{eq:IM_conn_1}), and (\ref{eq:IM_conn_2}), we find
\begin{equation}\label{eq:IM_conn_-}
\rho_A\left( \nabla_X^A a \right) = \nabla_X^M \rho_A(a).
\end{equation}
Similarly, computing $\mathcal F ([a, fb]_A)$ for all $f$, via (\ref{eq:IM_conn_3}), and using (\ref{eq:IM_conn_2}), we find
\begin{equation}\label{eq:IM_conn_+}
\mathcal L_{\rho_A (a)} \nabla^M =  \rho_A \circ \mathcal F (a).
\end{equation}
Finally, from (\ref{eq:IM_conn_3}), we see easily that (\ref{eq:IM_tensor_1}) holds. We stress that, in view of Remark \ref{rem:relations}, the latter identity, together with (\ref{eq:IM_conn_4}) and (\ref{stanco}), shows that the IM torsion of an IM connection is an IM vector valued $2$-form.
\end{remark}

\begin{remark} Similarly as for vector valued forms, there is a notion of \emph{compatibility} between a connection $\nabla$ in the tangent bundle $TA \to A$ of a vector bundle $A \to M$ (not necessarily a Lie algebroid) and the vector bundle structure. In order to avoid confusing terminology we call such compatibility \emph{fiber-wise linearity} (instead of simply \emph{linearity} as for vector valued forms). Namely, a connection $\nabla$ in $TA \to A$ is \emph{fiber-wise linear} if it is multiplicative wrt the Lie groupoid structure on $A$ with source and target the projection $A \to M$ and multiplication the fiber-wise addition. In this simple case, it is easy to see, e.g. in local coordinates, that the Formulas (\ref{eq:sec_comp_lin_mult}) establish a well-defined bijection between fiber-wise linear connections $\nabla$ in $TA \to A$ and $4$-tuples $(\mathcal F, \nabla^A, \nabla^M, l)$ as in Definition \ref{def:IM_conn} satisfying (\ref{eq:calF}) but not-necessarily satisfying (\ref{eq:IM_conn_1})--(\ref{eq:IM_conn_4}). For this reason, the latter $4$-tuples will be called the \emph{components of a fiber-wise linear connection}. Notice that, in the present case, for any section $a$ of $A$, the right invariant vector field $\overrightarrow a$ agrees with the vertical lift of $A$ and we also denote it by $a^\uparrow$.
 
 If $(\mathcal F, \nabla^A, \nabla^M, l)$ are the components of a fiber-wise linear connection $\nabla$, $\mathcal D$ is given by (\ref{eq:IM_torsion}) and $T^M$ is the torsion of $\nabla^M$, then $(\mathcal D, l, T^M)$ are the components of a linear vector valued $2$-form on $A$, and the latter agrees with the torsion of $\nabla$. We will often make use of this remark in the following without explicit comments. 
 \end{remark}
 
 \begin{theorem}\label{theor:IM_spray}
 Let $A \Rightarrow M$ be a Lie algebroid, let $TA \Rightarrow TM$ be its tangent prolongation Lie algebroid, and let $\nabla$ be a connection in $TA \to A$. Then the following conditions are equivalent 
 \begin{enumerate}
 \item $\nabla$ is fiber-wise linear and its components form an IM connection;
 \item the torsion $T^\nabla$ is a linear vector valued $2$-form and its components form an IM vector valued $2$-form, and, additionally, the geodesic spray $Z^\nabla$ is an IM vector field wrt the Lie algebroid $TA \Rightarrow TM$ \cite{MX1998}.
 \end{enumerate}
 \end{theorem}
 
There is a nice proof of Theorem \ref{theor:IM_spray} exploiting the graded geometry of the $Q$-manifold $A[1]$. As this proof requires a digression on connections on $Q$-manifolds, we postpone it to Appendix \ref{app:A}. Now, we state our main result.

\begin{theorem}\label{theor:Lie_theory}
 Let $\mathcal G \rightrightarrows M$ be a Lie groupoid, let $A \Rightarrow M$ be its Lie algebroid, and let $\nabla$ be a multiplicative connection in $ T\mathcal G  \to G$. Then $\nabla$ determines an IM connection $(\mathcal F, \nabla^A, \nabla^M, l)$ in $TA \to A$ via:
 \begin{equation}\label{eq:sec_comp_lin_mult}
 \begin{aligned}
 \mathcal F (a) & = P(\mathcal L_{\overrightarrow a} \nabla) \\
 \nabla^A a & = P(\nabla \overrightarrow a) \\
 \nabla^M & \text{is the $s$-projection of $\nabla$ (equivalently, its restriction to units)} \\
 l(a) & = P (\iota_{\overrightarrow a} T^\nabla)
 \end{aligned}
 \end{equation}
If $\mathcal G$ is source simply connected, then the assignment $\nabla \mapsto (\mathcal F, \nabla^A, \nabla^M, l)$ establishes a bijection between multiplicative connections in $ T\mathcal G  \to \mathcal G$ and IM connections in $TA \to A$ (see Theorem \ref{theor:Lie_mult_forms} for a definition of $P$).
 \end{theorem}

\begin{proof}
A linear connection in the tangent bundle is completely determined by its torsion and its geodesic spray. The main idea behind this proof is taking care of the two pieces of data separately. 

Let $\mathcal G \rightrightarrows M$, $A \Rightarrow M$ and $\nabla$ be as in the first part of the statement. The torsion $T^\nabla$ is a multiplicative vector valued $2$-form on $\mathcal G$. Denote by $(\mathcal D, l, \mathcal T^M)$ the corresponding IM vector valued $2$-form on $A$ and let $T^{\mathrm{lin}} \in \Omega^2 (A, TA)$ be the associated linear vector-valued $2$-form on $A$. It is easy to see that Formulas \ref{eq:sec_comp_lin_mult} define the components of a fiber-wise linear connection $\nabla^{\mathrm{lin}}$ in $TA \to A$ and that $T^{\mathrm{lin}}$ is exactly the torsion of $\nabla^{\mathrm{lin}}$.
In view of Theorem \ref{theor:IM_spray}, in order to conclude that $(\mathcal F, \nabla^A, \nabla^M, l)$ is an IM connection, it remains to show that the geodesic spray $Z^{\nabla^{\mathrm{lin}}}$ is an IM vector field wrt the Lie algebroid structure $TA \Rightarrow TM$. This follow from the following facts:
\begin{enumerate}
\item $Z^\nabla$ is a multiplicative vector field wrt the groupoid structure $ T\mathcal G  \rightrightarrows TM$ (Proposition \ref{prop:char_mult_conn});
\item $Z^{\nabla^{\mathrm{lin}}}$ is the corresponding IM vector field on the Lie algebroid $TA \Rightarrow TM$ (easy in local coordinates). 
\end{enumerate}

For the second part of the statement, assume that the Lie groupoid $\mathcal G \rightrightarrows M$ is source simply connected. In this case, the correspondence $\nabla \mapsto (\mathcal F, \nabla^A, \nabla^M, l)$ between multiplicative connections in $ T\mathcal G $ and IM connections in $TA$ can be inverted as follows. Notice preliminarily that the tangent prolongation $ T\mathcal G  \rightrightarrows TM$ is source simply connected as well. Begin with an $IM$ connection $(\mathcal F, \nabla^A, \nabla^M, l)$, and let $\nabla^{IM}$ be the associated fiber-wise linear connection in $TA \to A$. Denote by $(\mathcal D, l, T^M)$ the IM torsion of $(\mathcal F, \nabla^A, \nabla^M, l)$ and let $Z^{IM}$ be the geodesic spray of $\nabla^{IM}$. As $(\mathcal D, l, T^M)$ are exactly the components of the torsion of $\nabla^{IM}$, from Theorem \ref{theor:IM_spray} they are an IM vector valued $2$-form on $A$. Hence, from \cite[Theorem 3.19]{BD2019}, $(\mathcal D, l, T^M)$ can be integrated to a multiplicative vector valued $2$-form $T$ on $\mathcal G$. It remains to take care of the geodesic spray. From Proposition \ref{theor:IM_spray} again $Z^{IM}$ is an IM vector field with respect to the Lie algebroid structure $TA \Rightarrow TM$. Hence it can be integrated to a multiplicative vector field $Z$ on $ T\mathcal G$. If $Z$ is a spray, then it is the geodesic spray of a unique multiplicative symmetric connection $\nabla$ in $ T\mathcal G $, and $\nabla + T$ is the multiplicative connection we are looking for, i.e. $(\mathcal F, \nabla^A, \nabla^M, l)$ is the IM connection corresponding to $\nabla + T$.

So, it remains to prove that $Z$ is a spray. This means that 
\begin{enumerate}
\item $Z$ is of degree $1$ wrt the action $h^{\mathcal G}$ of $(\mathbb R, \cdot)$ on $ T\mathcal G $ given by fiber-wise scalar multiplication, i.e. $(h^{\mathcal G}_t)^\ast  Z = tZ$ for all $t \neq 0$, and
\item the vertical endomorphism $V_{\mathcal G} : T T\mathcal G  \to T T\mathcal G $ maps $Z$ to the Euler vector field $\mathsf{Eul}_{\mathcal G}$ on $ T\mathcal G $. 
\end{enumerate}

For item (1) notice that $h^{\mathcal G}$ is an action by groupoid maps, and it differentiates to the action $h^A$ by algebroid maps given by fiber-wise scalar multiplication in the fibers of $TA \to A$. As $Z^{IM}$ is a spray, we have $(h^A_t)^\ast Z^{IM} = t Z^{IM}$ for all $t \neq 0$. Hence $(h^A_t)^\ast Z^{IM} = t Z^{IM}$ is an IM vector field integrating, on one side, to $(h^{\mathcal G}_t)^\ast Z$, and, on the other side, to $t Z$. From the uniqueness of the integration, it follows that $(h^{\mathcal G}_t)^\ast Z = tZ$ for all $t \neq 0$, as desired.

For item (2) we argue as follows. First of all, it is easy to see that the vertical endomorphism $V_{\mathcal G} : T T\mathcal G  \to T T\mathcal G $ is a groupoid map differentiating to a Lie algebroid map: the vertical endomorphism $V_A : TTA \to TTA$ of $TA$. Similar considerations hold for the Euler vector field. Namely, the Euler vector field $\mathsf{Eul}_{\mathcal G}$ is a multiplicative vector field differentiating to an IM vector field $\mathsf{Eul}_A$: the Euler vector field on $TA$. As $Z$ is a multiplicative vector field (and $V_{\mathcal G}$ is a groupoid map), then $V_{\mathcal G} (Z)$ is a multiplicative vector field. The associated IM vector field is $V_A (Z^{IM}) = \mathsf{Eul}_A$. From the uniqueness of the integration, it follows that $V_{\mathcal G} (Z) = \mathsf{Eul}_{\mathcal G}$, and this concludes the proof.
 
\end{proof}

\section{Examples of Multiplicative and IM Connections}\label{Sec:Examples}

Not all Lie groupoids (resp., algebroids) support a multiplicative (resp., IM) connection. For instance, Equation (\ref{eq:IM_conn_-}) says that the anchor $\rho_A$ is parallel wrt to the connection induced by $\nabla^M, \nabla^A$ on the tensor algebra of $TM \oplus A$. It follows that $\rho_A$ has constant rank and the Lie algebroid $A$ is regular. Hence non-regular Lie groupoids (resp., algebroids) cannot support a multiplicative (resp., IM) connection. The existence of a multiplicative connection poses even more restrictions on a Lie groupoid. Namely, it follows from (\ref{eq:IM_conn_1}) that the Lie bracket $[-,-]_A$ vanishes on the kernel of the anchor $\rho_A$. Hence the isotropy Lie algebra bundle of a Lie groupoid supporting a multiplicative connection must be abelian. For instance a Lie group with non-abelian Lie algebra (resp., a non-abelian Lie algebra) cannot support a multiplicative (resp., IM) connection. Concluding, Lie groupoids with a multiplicative connections are not abundant and finding examples is not an easy task. However, non-trivial examples exist as we aim to show in this section. We begin with some simple examples of multiplicative connections on Lie groupoids.

\begin{example}[Orbifold groupoids]
Recall that a Lie groupoid $\mathcal G \rightrightarrows M$ is an \emph{\'etale groupoid} if $\dim \mathcal G = \dim M$, hence both the source and the target are local diffeomorphisms. Let $\mathcal G \rightrightarrows M$ be an \'etale groupoid, let $\nabla$ be a connection in $ T\mathcal G $, and let $\nabla^M$ be a connection in $TM$. As both $s,t$ are local diffeomorphisms, $\nabla$ is $(s,t)$-projectable with $(s,t)$-projection $\nabla^M$ iff $\nabla$ agrees with both the pull-back of $\nabla^M$ along $s$ and $t$: $\nabla = s^\ast \nabla^M = t^\ast \nabla^M$. In this case, $\nabla$ is completely determined by its $(s,t)$-projection. Now let $\nabla$ be $(s,t)$-projectable. Then $\nabla$ is multiplicative iff
\begin{equation}\label{eq:mult_etale}
m^\ast \nabla \alpha = \nabla^{(2)} m^\ast \alpha
\end{equation}
for all $\alpha \in \Omega^1 (\mathcal G)$. Actually, it is clearly enough that (\ref{eq:mult_etale}) holds for $\alpha = s^\ast \beta$ with $\beta \in \Omega^1 (M)$. So, let $\beta \in \Omega^1 (M)$, and compute
\begin{equation}\label{eq:comput_mult_etale}
\begin{aligned}
m^\ast \nabla s^\ast \beta & = m^\ast s^\ast \nabla^M \beta = \mathrm{pr}_2^\ast s^\ast \nabla^M \beta = \mathrm{pr}_2^\ast \nabla s^\ast \beta  = \nabla^{(2)}  \mathrm{pr}_2^\ast s^\ast \beta \\
& = \nabla^{(2)} m^\ast s^\ast \beta,
\end{aligned}
\end{equation}
where we used that $s \circ m = s \circ \mathrm{pr}_2 : \mathcal G^{(2)} \to \mathcal G$ and that $\nabla^{(2)}, \nabla$ are $\mathrm{pr}_2$-related.

The computation (\ref{eq:comput_mult_etale}) shows that Equation \ref{eq:mult_etale} is identically satisfied by any $(s,t)$-projectable connection in $ T\mathcal G $. We conclude that multiplicative connections in the tangent bundle of an \'etale groupoid $\mathcal G \rightrightarrows M$ are in one-to-one correspondence via pull-back with connections $\nabla^M$ in $TM$ such that 
\[
s^\ast \nabla^M = t^\ast \nabla^M.
\]
The latter condition says that $\nabla^M$ is invariant under the action of $\mathcal G$ on $M$. In particular, if $\mathcal G$ is an \emph{orbifold groupoid}, i.e. it is \'etale and proper, so that the orbispace $M / \mathcal G$ is an orbifold, then $\nabla^M$ descends to a connection in the tangent bundle of $M / \mathcal G$. We conclude that a multiplicative connection in the tangent bundle of an orbifold groupoid induces a connection in the tangent bundle of the associated orbifold in a natural way.
\end{example}

\begin{example}[Submersion groupoids]\label{ex:submersion}
Let $\pi : M \to B$ be a surjective submersion. The associated \emph{submersion groupoid} is the Lie subgroupoid $M \times_B M \rightrightarrows M$ of the pair groupoid $M \times M \rightrightarrows M$ of $M$. For this reason, abusing the notation, we will denote the structure maps of $M \times_B M$ and $M \times M$ in the same way. We will also denote by
\[
j : M \times_B M \hookrightarrow M \times M
\]
the inclusion. The manifold $(M \times_B M)^{(2)}$ of composable arrows in the submersion groupoid is diffeomorphic to $M \times_B M \times_B M$ via
\[
M \times_B M \times_B M \to (M \times_B M)^{(2)}, \quad (x,y,z) \mapsto ((x,y), (y,z)).
\]

Now, let $\nabla$ be a connection in the tangent bundle of $M \times_B M$. First of all we want to show that $\nabla$ is $(s,t)$-projectable with $(s,t)$-projection $\nabla^M$ iff the product connection $\nabla^M \times \nabla^M$ restricts to $M \times_B M$ and its restriction is $\nabla$, i.e.
\begin{equation}\label{eq:project_subm}
j^\ast \left( \nabla^M \times \nabla^M \right) \alpha = \nabla j^\ast \alpha
\end{equation}
for all $\alpha \in \Omega^1(M \times M)$. To do that first assume that $\nabla$ is $(s,t)$-projectable. It suffices to check that (\ref{eq:project_subm}) holds when $\alpha$ is of the form $s^\ast \beta$ or $t^\ast \beta$ for some $\beta \in \Omega^1 (M)$. So compute
\[
j^\ast \left(\nabla^M \times \nabla^M\right) s^\ast \beta =
j^\ast s^\ast \nabla^M \beta =\nabla j^\ast s^\ast \beta,  
\]
where we used that $\nabla^M \times \nabla^M$ and $\nabla^M$ are $s$-related. Similarly, (\ref{eq:project_subm}) holds when $\alpha = t^\ast \beta$. The fact that (\ref{eq:project_subm}) implies that $\nabla$ is $(s,t)$-projectable can be proved going back along the same computations. We conclude that when $\nabla$ is $(s,t)$-projectable, then it is completely determined by its projection $\nabla^M$.

Let $\nabla$ be $(s,t)$-projectable. Denote by 
\[
I: M \times_B M \times_B M \to M \times M \times M
\]
the natural embedding. Then it follows from (\ref{eq:project_subm}), that $\nabla^{(2)}$ and $\nabla^M \times \nabla^M \times \nabla^M$ are $I$-related. Now $\nabla$ is multiplicative iff
\begin{equation}\label{eq:mult_subm}
m^\ast \nabla \alpha = \nabla^{(2)} m^\ast \alpha
\end{equation}
holds for all $\alpha \in \Omega^1 (M \times_B M)$, and, actually, it is enough that (\ref{eq:mult_subm}) holds for $\alpha$ of either the form $j^\ast s^\ast \beta$ or $j^\ast t^\ast \beta$, with $\beta \in \Omega^1 (M)$. So, let $\beta \in \Omega^1 (M)$ and compute
\[
m^\ast \nabla j^\ast s^\ast \beta = m^\ast j^\ast \left(\nabla^M \times \nabla^M \right) s^\ast \beta = m^\ast j^\ast s^\ast \nabla^M \beta = (s \circ j \circ m)^\ast \nabla^M \beta,
\]
but the composition $s \circ j \circ m$ is the map
\[
s \circ j \circ m : M \times_B M \times_B M \to M, \quad (x,y,z) \mapsto x,
\]
hence it agrees with the composition
\[
M \times_B M \times_B M \overset{I}{\longrightarrow} M \times M \times M \overset{\mathrm{pr}_1}{\longrightarrow} M
\]
where $\mathrm{pr}_1$ is the projection onto the first factor. Therefore
\[
\begin{aligned}
m^\ast \nabla j^\ast s^\ast \beta & = (\mathrm{pr}_1 \circ I)^\ast \nabla^M \beta = I^\ast \mathrm{pr}_1^\ast \nabla^M \beta = I^\ast \left(\nabla^M \times \nabla^M \times \nabla^M \right) \mathrm{pr}_1^\ast \beta \\
& = \nabla^{(2)} I^\ast\mathrm{pr}_1^\ast \beta = \nabla^{(2)} m^\ast j^\ast s^\ast \beta
\end{aligned}
\]
where we used that $\nabla^{(2)}$ and $\nabla^M \times \nabla^M \times \nabla^M$ are $I$-related, and that $s \circ j \circ m = \mathrm{pr}_1 \circ I$ again. A similar computation shows that (\ref{eq:mult_subm}) holds true when $\alpha$ is of the form $j^\ast t^\ast \beta$. In other words every $(s,t)$-projectable connection in the tangent bundle of a submersion groupoid is automatically multiplicative. 

So multiplicative connections in $T (M \times_B M)$ are in one-to-one correspondence with connections $\nabla^M$ in $TM$ such that $\nabla^M \times \nabla^M$ restricts to $M \times_B M$. Finally, we want to re-express the latter condition in a more transparent way. We have the following
\begin{lemma}
Let $\nabla^M$ be a connection in the tangent bundle $TM$ of the total space $M$ of a surjective submersion $\pi : M \to B$. Then $\nabla^M \times \nabla^M$ restricts to $M \times_B M$ iff $\nabla^M$ projects to $B$, i.e. there exists a connection $\nabla^B$ in $TB$ such that $\nabla^M, \nabla^B$ are $\pi$-related.
\end{lemma}

\begin{proof}
For the ``only if'' part of the statement, it is enough to show that, for any $\beta \in \Omega^1 (B)$, the covariant $2$-tensor $\nabla^M \pi^\ast \beta$ is $\pi$-\emph{basic}, i.e. there exists a covariant $2$-tensor $\mathcal T^B$ on $B$ such that $\nabla^M \pi^\ast \beta = \pi^\ast \mathcal T^B$ (indeed, in this case, the assignment $\beta \mapsto \mathcal T^B$, $\beta \in \Omega^1 (B)$, is the connection $\nabla^B$ in $TB$ such that $\nabla^M, \nabla^B$ are $\pi$-related). To do that we use the following simple fact: a covariant $2$-tensor $\mathcal T$ on $M$ is $\pi$-basic iff $s^\ast \mathcal T = t^\ast \mathcal T$. So compute
\[
s^\ast \nabla^M \pi^\ast \beta = \nabla s^\ast \pi^\ast \beta = \nabla t^\ast \pi^\ast \beta = t^\ast \nabla^M \pi^\ast \beta,
\]
where we used that $\pi \circ s = \pi \circ t$.

The ``if part'' is a special case of Lemma \ref{lem:fibered}.
\end{proof}

We conclude that multiplicative connections in $T (M \times_B M)$ are in one-to-one correspondence with connections $\nabla^M$ in $TM$ that project to $B$. In particular, any multiplicative connection in the tangent bundle of a submersion groupoid $M \times_B M \rightrightarrows M$ uniquely determines a connection in the tangent bundle $TB$ of the associated orbispace $B$. Generic foliation groupoid can be treated in a similar way.
\end{example}

\begin{example}[Lie groups with abelian Lie algebra]
Let $\mathcal G$ be a Lie group, and let $\mathfrak g$ be its Lie algebra. We interpret $\mathcal G$ as a Lie groupoid over a point. There are at least two canonical (flat) connections in $ T\mathcal G  \to G$: the connections $\overrightarrow D, \overleftarrow D$ corresponding to the trivializations $ T\mathcal G  \cong \mathcal G \times \mathfrak g$ induced by right, left translations respectively. We already remarked that if $\mathfrak g$ is not abelian, $\mathcal G$ cannot possess a multiplicative connection in its tangent bundle. On the other hand, if $\mathfrak g$ is abelian, clearly $\overrightarrow D = \overleftarrow D$, it is a symmetric connection, and it is easy to see that it is a multiplicative connection. Actually, it is the only multiplicative connection in $\mathcal G$. Indeed, by dimension reasons, there are no nontrivial IM $(2,1)$ tensors on $\mathfrak g$, hence there are no nontrivial multiplicative $(2,1)$ tensors on $ T\mathcal G  \to \mathcal G$.
\end{example}

We now come to examples of Lie algebroids $A$ with an IM connection in their tangent bundle. When $A$ is integrable, its source simply connected integration provides an example of a Lie groupoid with a multiplicative connection (in its tangent bundle). We begin describing the infinitesimal version of Example \ref{ex:submersion}.

\begin{example}[The vertical bundle of a fibration]
Let $\pi : M \to B$ be a surjective submersion. We assume for simplicity that the fibers of $\pi$ are connected. The vertical bundle $T^\pi M$ is the Lie algebroid of the submersion groupoid $M \times_B M$. We want to describe IM connections in the tangent bundle of $T^\pi M$. Although the submersion groupoid is not source symply connected unless the fibers of $\pi$ are simply connected, our analysis will confirm that the IM connections in the tangent bundle of $T^\pi M$ are all obtained by differentiating the multiplicative connections in Example \ref{ex:submersion}. 

Let $(\mathcal F, \nabla^{T^\pi M}, \nabla^M, l)$ be an IM connection in the tangent bundle of $T^\pi M$. From (\ref{eq:IM_conn_-}), the connection $\nabla^M$ restricts to $T^\pi M$ and its restriction agrees with $\nabla^{T^\pi M}$. Now, let $V \in \mathfrak X^\pi (M)$, and $X \in \mathfrak X (M)$. Then we have
\[
\nabla^M_V X = \nabla_X^M V + [V, X] + T^M (V,X) = \nabla_X^M V + [V, X] + l(V)(X),
\]
where we used (\ref{eq:IM_tensor_5}). In particular, when $X$ is $\pi$-projectable, then $[V, X]$ is $\pi$-vertical so,
\begin{equation}\label{eq:nabla_VX}
\text{if $X$ is $\pi$-projectable}  \Rightarrow \nabla^M_V X \in \mathfrak X^\pi (M).
\end{equation}

From (\ref{eq:IM_conn_+}), the Lie derivative $\mathcal L_V \nabla^M$ must take values in $T^\pi M$ for all $\pi$-vertical vector fields $V$, i.e.
\[
[V, \nabla^M_X Y] - \nabla^M_{[V,X]}Y - \nabla^M_X [V,Y] \in \mathfrak X^\pi (M), 
\]
for all $V \in \mathfrak X^\pi (M)$, and all $X, Y \in \mathfrak X (M)$. When $X, Y$ are $\pi$-projectable, then $[V, X], [V,Y] \in \mathfrak X^\pi (M)$ and, from (\ref{eq:nabla_VX}), $\nabla_V X \in \mathfrak X^\pi (M)$ as well, so
\[
[V, \nabla^M_X Y] \in \mathfrak X^\pi (M).
\]
As the fibers of $\pi$ are connected, it follows from the arbitrariness of $V$ that $\nabla^M_X Y$ is a $\pi$-projectable vector field for all $\pi$-projectable vector fields $X, Y \in \mathfrak X (M)$. In other words $\nabla^M$ is a $\pi$-projectable connection. Finally notice that, in the present situation, (\ref{eq:IM_conn_+}) and (\ref{eq:IM_tensor_5}) determine $\mathcal F$ and $l$ completely from $\nabla^M$, and a direct computation reveals that they satisfy automatically all remaining axioms of an IM connection. We leave the straightforward details to the reader. We conclude that IM connections in the tangent bundle of $T^\pi M$ are precisely those of the form $(\mathcal F, \nabla^{T^\pi M}, \nabla^M, l)$, where
\begin{enumerate}
\item $\nabla^M$ is a $\pi$-projectable connection in $TM \to M$,
\item $\nabla^{T^\pi M}$ is the restriction of $\nabla^M$ to $T^\pi M$ (a $\pi$-projectable connection does always restrict to $T^\pi M$),
\item $\mathcal F$ is given by (\ref{eq:IM_conn_+}) (with $\rho_A : T^\pi M \to TM$ being the inclusion), and
\item $l$ is given by (\ref{eq:IM_tensor_5}), where $T^M$ is the torsion of $\nabla^M$.
\end{enumerate}
In particular, $(\mathcal F, \nabla^{T^\pi M}, \nabla^M, l)$ is completely determined by $\nabla^M$. Finally, it is easy to see that these IM connections are exactly those obtained by differentiating multiplicative connections in the submersion groupoid $M \times_B M$, as originally announced.
\end{example}

\begin{example}[Transitive Lie algebroids]
In this example we show that every transitive Lie algebroid with abelian isotropies can be equipped with an IM connection in its tangent bundle. Given a transitive Lie algebroid $A \Rightarrow M$, there is a short exact sequence of Lie algebroids
\begin{equation}\label{eq:SES_transitive}
0 \to K \to A \overset{\rho_A}{\to} TM \to 0
\end{equation}
where $K = \ker \rho_A \subseteq A$ is the isotropy bundle. The Lie algebroid $K$ is a bundle of Lie algebras and, from now on, we assume that it is a bundle of abelian Lie algebras. Let us now recall from \cite{mackenzie} Mackenzie's classification of transitive Lie algebroids in the special case when $K$ is abelian. We choose once for all a splitting $TM \to A$ of the exact sequence (\ref{eq:SES_transitive}) and use it to identify $A$ with the direct sum $TM \oplus K$ and the anchor $\rho_A$ with the projection $\mathrm{pr}_{TM} : TM \oplus K \to TM$ onto the first factor. Then the Lie bracket on $\Gamma (A) = \mathfrak X (M) \oplus \Gamma (K)$ becomes
\begin{equation}\label{eq:trans_abel}
\left[(X, h), (Y, k)\right] = \left([X,Y], \nabla^K_X k - \nabla^K_Y h - C(X,Y) \right),
\end{equation}
and one can show that
\begin{enumerate}
\item $\nabla^K$ is a connection in $K$,
\item $C \in \Omega^2 (M, K)$ is a $K$-valued $2$-form on $M$.
\end{enumerate}
Additionally, it follows from the Jacobi identity for the brackets on $\Gamma (A)$ (and the fact that $K$ is abelian) that
\begin{enumerate}
\item[(i)] $\nabla^K$ is a flat connection,
\item[(ii)] $C$ is a closed form wrt the connection differential $d^{\nabla^K} : \Omega^\bullet (M, K) \to \Omega^{\bullet + 1} (M, K)$ induced by $\nabla^K$.
\end{enumerate}
Conversely, given a vector bundle $K \to M$ with a flat connection $\nabla^K$ and a $d^{\nabla^K}$-closed $2$-form $C \in \Omega^2 (M, K)$, Formula (\ref{eq:trans_abel}) defines a Lie algebroid bracket on $TM \oplus K$, with anchor $\mathrm{pr}_K : TM \oplus K \to TM$, such that the isotropy bundle $K$ is abelian.

Now, let $\nabla^K, C$ be as above and consider the induced transitive Lie algebroid $TM \oplus K \Rightarrow M$. We look for IM connections $(\mathcal F, \nabla^{\oplus}, \nabla^M, l)$ in the tangent bundle of $TM \oplus K$. 
First of all, from (\ref{eq:IM_conn_-}) and (\ref{eq:IM_conn_+}), there exist sections $\theta, \lambda \in \Gamma (T^{2,0} M \otimes K)$ such that
\begin{equation}\label{eq:vai2}
\nabla^\oplus_X (Y,0) = (\nabla^M_X Y, \theta (X,Y)), \quad \text{and} \quad l(X,0)(Y) = (\mathcal T^M (X,Y), \lambda (X,Y))
\end{equation}
for all $X,Y \in \mathfrak X (M)$. 
Next, from (\ref{eq:IM_conn_1}) and (\ref{eq:trans_abel}), we have
\begin{equation}\label{eq:vai}
\left([X,Y], \nabla_X^K k - \nabla^K_Y h - C(X,Y) \right) = \nabla^\oplus_X(Y,k) - \nabla^\oplus_Y (X,h) - l(X,h)(Y)
\end{equation}
for all $X,Y \in \mathfrak X (M)$, and all $h, k \in \Gamma (K)$. When $Y = h = 0$ we get
\begin{equation}\label{eq:vai3}
\nabla_X^\oplus (0, k) = (0, \nabla^K_X k).
\end{equation}
Using the latter identity in (\ref{eq:vai}) we find that
\begin{equation}\label{eq:vai4}
l(0,h)(Y) = 0,
\end{equation}
and, using (\ref{eq:vai2}),
\begin{equation}\label{eq:vai5}
\lambda (X,Y) = \theta (X,Y) - \theta(Y,X) + C(X,Y).
\end{equation}
Now, using (\ref{eq:vai2}), (\ref{eq:vai3})--(\ref{eq:vai5}) in (\ref{eq:IM_conn_2}) gives the following formulas for $\mathcal F$:
\begin{equation}\label{eq:vai6}
\begin{aligned}
& \mathcal F (X,0)(Y,Z) \\
& = \Big(
\mathcal L_X \nabla^M (Y,Z), \nabla_X^K \theta (Y,Z) - \theta([X,Y], Z) - \theta (Y, [X, Z]) + \nabla^K_Y C(X, Z) - C (X, \nabla^M_Y Z)
 \Big)
\end{aligned}
\end{equation}
and
\begin{equation}\label{eq:vai7}
\mathcal F (0,k) (Y, Z) = \left( 0, \nabla^K_Y \nabla^K_Z k - \nabla^K_{\nabla^M_Y Z} k \right)
\end{equation}
for all $X,Y,Z \in \mathfrak X (M)$, and $k \in \Gamma (K)$. In particular, $(\mathcal F, \nabla^{\oplus}, \nabla^M, l)$ is completely determined by the sole data $\nabla^M, \theta$. 

Conversely, let $\nabla^M$ be a connection in $TM \to M$ with torsion denoted $\mathcal T^M$, and let $\theta \in \Gamma (T^{2,0} M \otimes K)$. Then we can define a connection $\nabla^\oplus$ in $TM \oplus K$ via the first one of (\ref{eq:vai2}) and (\ref{eq:vai3}). We can also define a vector bundle map $l : TM \oplus K \to TM \otimes \left(TM \oplus K \right)$ via the second one of (\ref{eq:vai2}), (\ref{eq:vai4}) and (\ref{eq:vai5}). Finally, we can define an operator $\mathcal F : \Gamma (TM \oplus K) \to \Gamma \left(T^{2,0} M \otimes \left(TM \oplus K\right)\right)$ via (\ref{eq:vai6}) and (\ref{eq:vai7}). By construction the data $(\mathcal F, \nabla^\oplus, \nabla^M, l)$ now satisfy (\ref{eq:IM_conn_1}) and (\ref{eq:IM_conn_2}). A long but straightforward computation reveals that they also satisfy (\ref{eq:calF}), (\ref{eq:IM_conn_3}) and (\ref{eq:IM_conn_4}). We leave the details to the reader. We conclude that the assignment $(\mathcal F, \nabla^{\oplus}, \nabla^M, l) \mapsto (\nabla^M, \theta)$ establishes a one-to-one correspondence between IM connections in the tangent bundle of $TM \oplus K$ (equipped with the Lie algebroid structure determined by $\nabla^K$ and $C$) and pairs consisting of a connection $\nabla^M$ in $TM \to M$ and a section $\theta \in \Gamma (T^{2,0} M \otimes K)$, with no further restrictions. Hence, all transitive Lie algebroids (resp., source simply connected groupoid) with abelian isotropy Lie algebra bundle possess an IM connection (resp., a multiplicative connection) in their tangent bundle.
\end{example}

We conclude this section presenting a (family of) toy example(s) of a (regular) Lie algebroid $A \Rightarrow M$ admitting an IM connection in its tangent bundle $TA \to A$ and whose anchor is neither injective nor surjective. Notice that solving the equations (\ref{eq:IM_conn_1})--(\ref{eq:IM_conn_4}) for both a Lie algebroid structure and an IM connection in full generality is an hard problem. In order to find a simple but non-trivial example we find it useful to adopt simplifying ansatzes. In what follows, we indeed arrive at the concrete example through a hierarchy of such ansatzes, that we have decided to discuss as they might have an independent interest. As a first simplifying hypothesis, we will search for \emph{symmetric IM connections}, i.e. an IM connection whose IM torsion vanishes. This is actually not restrictive as a Lie algebroid possesses an IM connection iff it possesses a symmetric IM connection.

Begin with a plain vector bundle $A \to M$, and let $(\mathcal F, \nabla^A, \nabla^M, l)$ be the components of a symmetric fiber-wise linear connection in $TA \to A$. In this case $\mathcal D, l$ and the torsion of $\nabla^M$ vanish identically, hence the second order differential operator $\mathcal F$ takes values in symmetric tensors: $\mathcal F : \Gamma (A) \to \Gamma (S^2 T^\ast M \otimes A)$. Additionally, the identity (\ref{eq:calF}) simplifies to
\begin{equation}\label{eq:calF_symm}
\mathcal F (fa) = f \mathcal F (a) + \nabla^M df \otimes a + df \odot \nabla^A a, \quad \text{for all $f \in C^\infty (M)$, and $a \in \Gamma (A)$}.
\end{equation}
Notice that the second summand $\nabla^M df \otimes a$ is duly a symmetric tensor when $\nabla^M$ is a symmetric connection.

 There is another description of $(\mathcal F, \nabla^A, \nabla^M, l)$ that we now explain. First of all, we have the following 

\begin{proposition}\label{prop:calF_0}
An operator $\mathcal F :  \Gamma (A) \to \Gamma (S^2 T^\ast M \otimes A)$ satisfies Equation (\ref{eq:calF_symm}) iff there exists a (necessarily unique) bundle map $F : A \to S^2 T^\ast M \otimes A$ such that
\begin{equation}\label{eq:F}
\mathcal F = \mathcal F_0 + F,
\end{equation}
where $\mathcal F_0 : \Gamma (A) \to \Gamma (S^2 T^\ast M \otimes A)$ is the operator given by
\[
\mathcal F_0 (a) (X, Y) = \frac{1}{2} \left(\nabla^A_X \nabla^A_Y - \nabla^A_{\nabla^M_X Y} + \nabla^A_Y \nabla^E_X - \nabla^A_{\nabla^M_Y X} \right) a
\]
for all $a \in \Gamma(A)$, $X, Y \in \mathfrak X (M)$.
\end{proposition}

\begin{proof}
A direct computation shows that
\[
\mathcal F_0 (fa) = f \mathcal F_0 (a) + \nabla^M df \otimes a + df \odot \nabla^A e.
\]
It follows that $F := \mathcal F - \mathcal F_0$ is actually a vector bundle map $A \to S^2 T^\ast M \otimes A$. This concludes the proof.
\end{proof}

We stress that the operator $\mathcal F_0$ of Proposition \ref{prop:calF_0} does only depend on $\nabla^A, \nabla^M$.

\begin{corollary}\label{cor:tertiary_comp}
The assignment $(\mathcal F, \nabla^A, \nabla^M, 0) \mapsto (\nabla^A, \nabla^M, F)$ establishes a $C^\infty (M)$-linear bijection between the components of symmetric fiber-wise linear connections in $TA \to A$ and triples $(\nabla^A, \nabla^M, F)$ consisting of 
\begin{enumerate}
\item a connection $\nabla^A$ in $A \to M$,
\item a symmetric connection $\nabla^M$ in $TM \to M$,
\item a vector bundle map $F : A \to S^2 T^\ast M \otimes A$,
\end{enumerate}
with no further restrictions. Here $F$ is given by (\ref{eq:F}) in terms of the components $(\mathcal F, \nabla^A, \nabla^M, 0)$.
\end{corollary}

We call $(\nabla^A, \nabla^M,  F)$ as in Corollary \ref{cor:tertiary_comp} the \emph{secondary components} of a symmetric fiber-wise linear connection. Their advantage on the plain components is that there are no relations among them.

Now, let $A \Rightarrow M$ be a Lie algebroid. Our next aim is to find conditions on the secondary components $(\nabla^A, \nabla^M, F)$ of a symmetric fiber-wise linear connection in $TA \to A$ so that the corresponding (plain) components $(\mathcal F, \nabla^A, \nabla^M, 0)$ do actually form an IM connection. We will not do this in full generality. Instead, we will make our second simplifying hypothesis and assume that $\nabla^A$ is a flat connection. In what follows, we denote by $\bbnabla$ the connection induced by $\nabla^A, \nabla^M$ on the whole tensor algebra of $A \oplus TM$.

\begin{theorem}\label{theor:nabla^A_flat}
Let $A \Rightarrow M$ be a Lie algebroid, let $(\mathcal F, \nabla^A, \nabla^M, 0)$ be the components of a symmetric fiber-wise linear connection in $TA \to A$, and let $(\nabla^A, \nabla^M, F)$ be the secondary components. Assume that $\nabla^A$ is a flat connection. Then $(\mathcal F, \nabla^A, \nabla^M, 0)$ is an IM connection iff
\begin{align}
[a,b]_A & = \nabla^A_{\rho_A (a)} b - \nabla^A_{\rho_A(b)} a, \label{eq:3.1} \\
\bbnabla \rho_A & = 0, \label{eq:3.2}\\
\rho_A \circ  F(a) & = \iota_{\rho_A(a)}R^M,  \label{eq:3.3}\\
\iota_{\rho_A (b)} F(a) & = 0,\label{eq:3.4}\\
\left(\bbnabla_{\rho_A (a)} F\right)(b), & =\left(\bbnabla_{\rho_A (b)} F\right)(a), \label{eq:3.5}
\end{align}
for all $a,b \in \Gamma (A)$, where $R^M$ is the curvature of $\nabla^M$. 
\end{theorem}

\begin{proof}
First of all, assume that $(\mathcal F, \nabla^A, \nabla^M, 0)$ is an IM connection. Equations (\ref{eq:3.1}) and (\ref{eq:3.2}) are just Equations (\ref{eq:IM_conn_1}) and (\ref{eq:IM_conn_-}) (in the symmetric connection case). Equation (\ref{eq:3.3}) can be easily obtained from (\ref{eq:IM_conn_+}), (\ref{eq:F}) and the fact that $\nabla^M$ is a symmetric connection. Equation (\ref{eq:3.4}) follows from (\ref{eq:IM_conn_2}), (\ref{eq:F}) and the hypothesis that the curvature of $\nabla^A$ vanishes. Finally, equation (\ref{eq:3.5}) can be obtained from (\ref{eq:IM_conn_3}) as follows. First of all, it is straightforward to check that the expression
\[
\mathcal F \left([a,b]_A\right) - \mathcal L_a \mathcal F(b) + \mathcal L_b \mathcal F(a)
\]
is $C^\infty (M)$-bilinear in the arguments $(a, b)$ modulo Equations (\ref{eq:IM_conn_-}) and (\ref{eq:IM_conn_2}). Both sides of (\ref{eq:3.5}) are $C^\infty (M)$-bilinear in $(a, b)$ as well. So, when $(\mathcal F, \nabla^A, \nabla^M, 0)$ is an IM connection, in order to check (\ref{eq:3.5}), it is enough to check that it holds when $a,b$ are $\nabla^A$-parallel sections. This is an easy computation that we leave to the reader. 

Conversely, that Equations (\ref{eq:3.1})--(\ref{eq:3.5}) imply Equations (\ref{eq:IM_conn_1})--(\ref{eq:IM_conn_4}) can be proved reversing the previous computations.
\end{proof}

\begin{corollary}\label{cor:F=0}
Let $A \to M$ be a vector bundle. Consider
\begin{itemize}
\item[(i)] a flat connection $\nabla^A$ in $A \to M$,
\item[(ii)] a symmetric connection $\nabla^M$ in $TM \to M$ with curvature $R^M$,
\item[(iii)] a vector bundle map $\rho_A : A \to TM$.
\end{itemize} 
Assume that
\begin{itemize}
\item[(I)] $\rho_A$ is parallel wrt the connection induced by $\nabla^A, \nabla^M$ in the tensor algebra of $A \oplus TM$, and
\item[(II)] $\iota_{\rho_A(a)} R^M = 0$ for all $a \in \Gamma (A)$. 
\end{itemize}
Then
\begin{enumerate}
\item The bracket $[a, b]_A = \nabla^A_{\rho_A(a)}b - \nabla^A_{\rho_A(b)}a$ on $\Gamma (A)$ gives to $A$ the structure of a Lie algebroid with anchor $\rho_A$, and
\item $(\nabla^A, \nabla^M, 0)$ are the secondary components of a symmetric IM connection in $TA \to A$. 
\end{enumerate}
\end{corollary}

\begin{proof}
It immediately follows from Theorem \ref{theor:nabla^A_flat} and the following remark: if $\nabla^A$ is a connection in $A \to M$ and $\rho_A : A \to TM$ is a bundle map, then the bracket $[a, b]_A = \nabla^A_{\rho_A(a)}b - \nabla^A_{\rho_A(b)}a$ on $\Gamma (A)$ is a biderivation with anchor $\rho_A$ and it satisfies the Jacobi identity iff
\[
R^A (\rho_A(a), \rho_A(b))c + R^A (\rho_A(b), \rho_A(c))a + R^A (\rho_A(c), \rho_A(a))b = 0
\]
where $R^A$ is the curvature of $\nabla^A$.
\end{proof}

\begin{example}[A further toy example]
We are finally ready to construct a simple example. We assume $M = G$ is a Lie group with Lie algebra $\mathfrak g$, $A =  T G $ is its tangent bundle. Additionally, we choose $\nabla^A = \overrightarrow{D}$ the flat connection corresponding to the trivialization $TM \cong G \times \mathfrak g$ induced by right translations. Clearly, $\overrightarrow D$ is the unique flat connections in $TM \to M$ such that all right invariant tensors are parallel. In particular the torsion $T^D$ is exactly $- \overrightarrow B$: the right invariant vector valued $2$-form that agrees with minus the Lie bracket $B = [-,-]_{\mathfrak g}$ of $\mathfrak g$ at the unit. We also put $\nabla^M = \overrightarrow D{}^{\mathrm{sym}} = \overrightarrow D - \frac{1}{2} \overrightarrow B$. The curvature of $\nabla^M$ is then given by
\begin{equation}\label{eq:R^M_Lie_group}
R^M (\overrightarrow v, \overrightarrow w) \overrightarrow z = - \frac{1}{4} \overrightarrow{\left[ \left[v,w\right]_{\mathfrak g}, z\right]_{\mathfrak g}}, \quad v,w,z \in \mathfrak g.
\end{equation}
Finally, we choose $\rho$ to be a right invariant $(1,1)$ tensor: $\rho = \overrightarrow r$, with $r \in \operatorname{End} \mathfrak g$. Then $\rho_A$ is $\bbnabla$-parallel iff $\rho \circ \overrightarrow B = 0$ or, equivalently,
\begin{equation}\label{eq:r1}
[ \mathfrak g, \mathfrak g] \subseteq \ker r.
\end{equation}
The last condition we want to satisfy is $\iota_{\rho_A(X)} R^M = 0$ for all $X \in \mathfrak X (M)$. But, from (\ref{eq:R^M_Lie_group}), this is equivalent to
\begin{equation}\label{eq:r2}
\left[ \left[r(v),w\right]_{\mathfrak g}, z\right]_{\mathfrak g} = 0
\end{equation} 
for all $v, w, z \in \mathfrak g$. This discussion, together with Corollary \ref{cor:F=0} proves the following
\begin{proposition}
Let $G$ be a Lie group with Lie algebra $\mathfrak g$, and let $r \in \operatorname{End} \mathfrak g$ be an endomorphism satisfying (\ref{eq:r1}) and (\ref{eq:r2}). Then the bracket
\[
[X,Y]_{r} = \overrightarrow D_{\overrightarrow r (X)} Y - \overrightarrow D_{\overrightarrow r (Y)} X
\]
gives to $TM \to M$ the structure $(TM)_r$ of a Lie algebroid with anchor $\overrightarrow r$, and $(\overrightarrow D, \overrightarrow D{}^{\mathrm{sym}}, 0)$ is a triple of secondary components of a symmetric IM connection in the tangent bundle of $(TM)_r$.
\end{proposition}
\end{example}

Beyond the abelian Lie algebra, the simplest possible Lie algebra $\mathfrak g$ possessing an endomorphism $r \in \operatorname{End} \mathfrak g$ as in the latter proposition is the Heisenberg Lie algebra $\mathfrak h = \langle e_1, e_2, e_3\rangle$, with
\[
[e_1, e_2]_{\mathfrak h} = [e_1, e_3]_{\mathfrak h} = 0, \quad [e_2, e_3]_{\mathfrak h} = e_1.
\]
Denote by $(e_1^\ast, e_2^\ast, e_3^\ast)$ the dual frame in $\mathfrak h^\ast$.
The derived Lie algebra $[\mathfrak h, \mathfrak h]$ is $1$-dimensional and it is spanned by $e_1$. Additionally every element $v \in \mathfrak h$ is $2$-step nilpotent, hence any $r$ of the form $r = v_2 \otimes e_2^\ast + v_3 \otimes e^\ast_3$, with $v_2, v_3 \in \mathfrak h$, satisfies both conditions  (\ref{eq:r1}) and (\ref{eq:r2}).

\appendix

\section{Proof of Theorem \ref{theor:obst_class}}\label{app:B}

We begin remarking that, as $\nabla$ is $s$-projectable, then the product connection $\nabla \times \nabla \times \nabla$ restricts to a unique connection $\nabla^{[3]}$ in $ T\mathcal G ^{[3]} \to \mathcal G^{[3]}$. We will also need the following maps:
\[
\mathfrak F_{i,j} := \bar m \circ (\mathrm{pr}_i \times \mathrm{pr}_j) : \mathcal G^{[3]} \to \mathcal G, \quad (g_1, g_2, g_3) \mapsto g_ig_j^{-1}, \quad i,j = 1,2,3.
\]
Clearly we have
\[
\bar m \circ \left(\mathfrak F_{1,3} \times \mathfrak F_{2,3}\right) = \mathfrak F_{1,2}.
\]

Now, consider the cochain $\bar \delta \mathsf{At}_\nabla : T^{0,2}\mathcal G^{[3]} \to  T\mathcal G $. According to (\ref{eq:delta_bar}) it is given by
\begin{equation}\label{eq:deltaT_nabla}
\begin{aligned}
 & \bar \delta \mathsf{At}_\nabla (v^\otimes_1, v^\otimes_2, v^\otimes_3) \\
 &  = - d \bar m \left( \mathsf{At}_\nabla (v^\otimes_1, v^\otimes_3), \mathsf{At}_\nabla (v^\otimes_2, v^\otimes_3) \right)   + \mathsf{At}_\nabla (v^\otimes_1, v^\otimes_2) \\  & \quad
 - \mathsf{At}_\nabla \left( d \bar m{}^{\otimes 2} (v^\otimes_1, v^\otimes_3) , d \bar m{}^{\otimes 2} (v^\otimes_2, v^\otimes_3) \right), \quad \quad (v^\otimes_1, v^\otimes_2, v^\otimes_3) \in T^{0,2} \mathcal G^{[3]}.
\end{aligned}
\end{equation}
But it can also be viewed as a module homomorphism
\[
\mathcal S_\nabla : \Omega^1 (\mathcal G) \to \mathfrak T^{0,2} (\mathcal G^{[3]})
\]
covering the algebra homomorphism $\mathfrak F^\ast_{1,2} : C^\infty (\mathcal G) \to C^\infty (\mathcal G^{[3]})$ via
\begin{equation}\label{eq:S_nabla}
\big\langle \mathcal S_\nabla (\theta), (v^\otimes_1, v^\otimes_2, v^\otimes_3) \big\rangle = \big\langle \theta, \bar \delta \mathsf{At}_\nabla (v^\otimes_1, v^\otimes_2, v^\otimes_3) \big\rangle, \quad \theta \in \Omega^1 (\mathcal G). 
\end{equation}
We want to show that $\mathcal S_\nabla = 0$. A direct computation exploiting (\ref{eq:S_nabla}), (\ref{eq:deltaT_nabla}) and the definition of $\mathsf{At}_\nabla$ shows that:
\[
\begin{aligned}
& \big\langle \mathcal S_\nabla (\theta), (v^\otimes_1, v^\otimes_2, v^\otimes_3) \big\rangle \\
& = \left\langle \left((\mathfrak F_{1,3}  \times \mathfrak F_{2,3})^\ast \circ \nabla^{[2]} \circ \bar m{}^\ast - \nabla^{[3]} \circ \mathfrak F_{1,2}^\ast \right) \theta , (v^\otimes_1, v^\otimes_2, v^\otimes_3) \right\rangle \\
&\quad  - \big\langle m{}^\ast \theta , \big(\mathsf{At}_\nabla (v^\otimes_1, v^\otimes_3), \mathsf{At}_\nabla (v^\otimes_2, v^\otimes_3) \big)\big\rangle
\end{aligned}.
\]
Hence
\[
\mathcal S_\nabla = \mathfrak U - \mathfrak V \circ \bar m{}^\ast
\]
where 
\[
\mathfrak U :=  (\mathfrak F_{1,3}  \times \mathfrak F_{2,3})^\ast \circ \nabla^{[2]} \circ \bar m{}^\ast - \nabla^{[3]} \circ \mathfrak F_{1,2}^\ast
\]
and $\mathfrak V : \Omega^1 (\mathcal G^{[2]}) \to \mathfrak T^{0,2} (\mathcal G^{[3]})$ is the module homomorphism covering the algebra homomorphism 
\[
(\mathfrak F_{1,3} \times \mathfrak F_{2,3})^\ast : C^\infty (\mathcal G^{[2]}) \to C^\infty (\mathcal G^{[3]})
\]
 given by
\[
 \big\langle \mathfrak V (\Theta), (v^\otimes_1, v^\otimes_2, v^\otimes_3) \big\rangle
 = \big\langle \Theta , \big(\mathsf{At}_\nabla (v^\otimes_1, v^\otimes_3), \mathsf{At}_\nabla (v^\otimes_2, v^\otimes_3) \big)\big\rangle, \quad \Theta \in \Omega^1 (\mathcal G^{[2]}).
\]
We have to prove that $\mathfrak U = \mathfrak V \circ \bar m{}^\ast$. To do this, first notice that $\mathfrak V$ is completely determined by the compositions $\mathfrak V \circ \mathrm{pr}_i^\ast : \Omega^1 (\mathcal G) \to \mathfrak T^{0,2} (\mathcal G^{[3]})$, $i = 1,2$. Now, for all $\theta \in \Omega^1 (\mathcal G)$, we have 
\[
\begin{aligned}
&\big\langle \mathfrak V (\mathrm{pr}_1^\ast \theta), (v^\otimes_1, v^\otimes_2, v^\otimes_3) \big\rangle \\
& = \big\langle \theta , \mathsf{At}_\nabla (v^\otimes_1, v^\otimes_3)\big\rangle
 =  \big\langle \left(\bar m{}^\ast \circ \nabla \theta - \nabla^{[2]} \circ \bar m{}^\ast \right) \theta , (v^\otimes_1, v^\otimes_3)  \big\rangle \\
 & = \big\langle \left(\mathrm{pr}_1 \times \mathrm{pr}_3\right)^\ast \left(\bar m{}^\ast \circ \nabla \theta - \nabla^{[2]} \circ \bar m{}^\ast \right) \theta, (v^\otimes_1, v^\otimes_2, v^\otimes_3)  \big\rangle \\
 & = \big\langle \left(\mathfrak F_{1,3}^\ast \circ \nabla \theta - \nabla^{[3]} \circ \mathfrak F_{1,3}^\ast  \right) \theta, (v^\otimes_1, v^\otimes_2, v^\otimes_3)  \big\rangle
 \end{aligned},
\]
hence
\[
\begin{aligned}
\mathfrak V \circ \mathrm{pr}_1^\ast & =  \mathfrak F_{1,3}^\ast \circ \nabla - \nabla^{[3]} \circ \mathfrak F_{1,3}^\ast \\
& = \left( \left(\mathfrak F_{1,3} \times \mathfrak F_{2,3}\right)^\ast \circ \nabla^{[2]} - \nabla^{[3]} \circ  \left(\mathfrak F_{1,3} \times \mathfrak F_{2,3}\right)^\ast \right) \circ \mathrm{pr}_1^\ast.
\end{aligned}
\]
Similarly
\[
\mathfrak V \circ \mathrm{pr}_2^\ast = \left( \left(\mathfrak F_{1,3} \times \mathfrak F_{2,3}\right)^\ast \circ \nabla^{[2]} - \nabla^{[3]} \circ  \left(\mathfrak F_{1,3} \times \mathfrak F_{2,3}\right)^\ast \right) \circ \mathrm{pr}_2^\ast .
\]
We conclude that 
\[
\mathfrak V = \left(\mathfrak F_{1,3} \times \mathfrak F_{2,3}\right)^\ast \circ \nabla^{[2]} - \nabla^{[3]} \circ  \left(\mathfrak F_{1,3} \times \mathfrak F_{2,3}\right)^\ast
\]
and therefore
\[
\begin{aligned}
\mathfrak V \circ \bar m{}^\ast & = \left(\mathfrak F_{1,3} \times \mathfrak F_{2,3}\right)^\ast \circ \nabla^{[2]} \circ \bar m{}^\ast - \nabla^{[3]} \circ  \left(\mathfrak F_{1,3} \times \mathfrak F_{2,3}\right)^\ast \circ \bar m{}^\ast \\
& = \left(\mathfrak F_{1,3} \times \mathfrak F_{2,3}\right)^\ast \circ \nabla^{[2]} \circ \bar m{}^\ast - \nabla^{[3]} \circ \mathfrak F_{1,2}^\ast = \mathfrak U.
\end{aligned}
\]
This proves (i).

Next take two $s$-projectable connections $\nabla, \nabla'$ in 
$ T\mathcal G  \to \mathcal G$. Their difference $\nabla - \nabla'$
 is an $s$-projectable $(2,1)$-tensor on $\mathcal G$, which can 
 be seen as a $0$-cochain $\mathcal U$ in $\bar C_{\mathrm{def}}
 (\mathcal G, T^{2,0})$. As mentioned in Remark \ref{rem:delta_2,1_tensors}, the differential $\bar \delta \mathcal U$, 
 seen as a module homomorphism $\Omega^1 (\mathcal G) 
 \to \mathfrak T^{2,0}(\mathcal G^{[2]})$, agrees with the combination
\[
\mathcal S = \bar m{}^\ast \circ \left(\nabla - \nabla'\right) - \left(\nabla - \nabla' \right)^{[2]} \circ \bar m{}^\ast.
\]
Now, from the easy formula
\begin{equation}\label{eq:12f34}
\big(\nabla - \nabla'\big) \times \big(\nabla - \nabla'\big) = \nabla \times \nabla - \nabla' \times \nabla',
\end{equation}
we get $(\nabla - \nabla')^{[2]} = \nabla^{[2]} - \nabla'{}^{[2]}$. Hence
\[
\mathcal S = \left( \bar m{}^\ast \circ \nabla - \nabla^{[2]} \circ \bar m{}^\ast\right) - \left(\bar m{}^\ast \circ \nabla' - \nabla'{}^{[2]} \circ \bar m{}^\ast\right).
\]
Putting everything together we conclude that
\[
\bar \delta \mathcal U = \mathsf{At}_\nabla - \mathsf{At}_{\nabla'}.
\]
This proves (ii).

Finally, suppose that there exists a multiplicative connection $\nabla$ in $ T\mathcal G  \to \mathcal G$. Then, from Proposition \ref{prop:mult_conn_div}, $\mathsf{At}_\nabla$, hence its cohomology class, vanishes. Conversely, the vanishing of the cohomology class $[\mathsf{At}_\nabla]$ means that there exists a $(2,1)$-tensor $U$ on $\mathcal G$ such that, when we see it as a map $U : \Omega^1 (\mathcal G) \to \mathfrak T^{2,0} (\mathcal G)$, we have
\[
\bar m{}^\ast \circ \nabla - \nabla^{[2]} \circ \bar m{}^\ast = \bar m{}^\ast \circ U - U^{[2]} \circ \bar m{}^\ast
\]
(see Remark  \ref{rem:delta_2,1_tensors} again). Hence, from (\ref{eq:12f34}) (in the case $\nabla' = \nabla - U$), we find
\[
\bar m{}^\ast \circ (\nabla-U) - (\nabla-U)^{[2]} \circ \bar m{}^\ast = 0,
\]
i.e. $\nabla - U$ is a multiplicative connection.

\section{IM Connections and Homological Vector Fields}\label{app:A}

In this appendix we provide a proof of Theorem (\ref{theor:IM_spray}) exploiting some features of connections on $\mathbb NQ$-manifolds. We assume the reader is familiar with the fundamentals of graded manifolds and calculus on them and we only recall that an $\mathbb N$-manifold of degree $n$ is a graded manifold $\mathcal M$ whose coordinates are concentrated in degree $0, \ldots, n$. If $\mathcal M$ is  additionally equipped with a \emph{$Q$-structure}, i.e. a degree $1$-vector field $Q$ such that $[Q,Q] = 0$, we call it a (degree $n$) $\mathbb N Q$-manifold. For instance, if $A \Rightarrow M$ is a Lie algebroid, then its de Rham differential $d_A$ is a $Q$-structure on $A[1]$, the graded manifold obtained from $A$ by shifting by one the degree in the fiber coordinates, and $(A[1], d_A)$ is a degree $1$ $\mathbb N Q$-manifold. All degree $1$ $\mathbb N Q$-manifolds arise in this way.

We begin with a discussion on tensors on degree $1$ $\mathbb N Q$-manifolds. Considering vector valued forms will suffice. Our convention on differential forms on graded manifolds is the following: a differential $k$-form on a graded manifold $\mathcal M$ is a graded skew-symmetric $C^\infty (\mathcal M)$-multilinear map
$
\mathfrak X (\mathcal M) \times \cdots \times \mathfrak X (\mathcal M) \to C^\infty (\mathcal M)
$. 
This means that the wedge product of forms obeys the following commutativity rule: if $\omega$ is a $k$-form of (internal, coordinate) degree $|\omega|$ and $\omega'$ is a $k'$-form of degree $|\omega'|$ then 
$
\omega \wedge \omega' = (-)^{|\omega||\omega'| + kk'} \omega' \wedge \omega
$.
Our convention on vector valued forms is similar. Now, let $\mathcal M$ be a degree $1$ $\mathbb N$-manifold. So $\mathcal M = A[1]$ for some ordinary (non graded) vector bundle $A \to M$. There are no non-trivial vector valued forms on $A[1]$ of degree $-2$ or lower. Similarly as in the ordinary case, degree $-1$ vector fields on $A[1]$ identify with sections of $A$ via the \emph{vertical lift} construction $\Gamma (A) \to \mathfrak X (A[1])$, $a \mapsto a^\uparrow$, where $a^\uparrow$ is the interior product with $a$ in $C^\infty (A[1]) = \Gamma (\wedge^\bullet A^\ast)$:
\[
a^\uparrow = \iota_a :  \Gamma (\wedge^\bullet A^\ast) \to  \Gamma (\wedge^{\bullet-1} A^\ast).
\] 
Degree $0$ vector fields on $A[1]$ identify with sections of the gauge algebroid $DA \Rightarrow M$. If $X$ is a degree $0$ vector field, the associated derivation $D$ is uniquely defined by
$[X, a^\uparrow] = (Da)^\uparrow$, for all $a \in \Gamma (A)$.
We will denote by $X_D$ the degree $0$ vector field on $A[1]$ corresponding to $D \in \Gamma (DA)$. If $(A[1], d_A)$ is the degree $1$ $\mathbb N Q$-manifold corresponding to a Lie algebroid $A \Rightarrow M$, then we have
$
[d_A, a^\uparrow] = X_{\operatorname{ad} a}$ for all $a \in \Gamma (A)$, where $\operatorname{ad} a$ is the derivation of $A \to M$ given by
$
\operatorname{ad} a (b) = [a, b]_A
$. In particular the symbol of the derivation $\operatorname{ad} a$ is $\sigma (\operatorname{ad} a) = \rho_A (a)$.

Degree $-1$ vector valued $p$-forms $\mathcal S$ on $A[1]$ identify with ordinary $A$-valued $p$-forms $S$ on $M$ via
\[
\mathcal S (X_{D_1}, \ldots, X_{D_p}) = S(\sigma(D_1), \ldots, \sigma (D_p))^\uparrow,
\]
for all $D_1, \ldots, D_p$ derivations of $A$ (use, e.g., local coordinates). 
In a very similar (and obvious) way $(p,1)$-tensors of degree $-1$ identify with $A$-valued $(p,0)$-tensors on $M$, i.e. sections of $T^{p,0} M \otimes A$.

Degree $0$ vector valued $p$-forms on $A[1]$ identify with linear vector valued $p$-forms on $A$. Indeed, the same formulas as for linear forms identify a degree $0$ form in $\Omega^p (A[1], TA[1])$ with a triple $(\mathcal D, l, \mathcal T^M)$ of \emph{components} hence with a linear form $\mathcal T \in  \Omega^p (A, TA)$. We will denote by $\mathcal T^{[1]}$ the degree $0$ vector valued form on $A[1]$ corresponding to a linear vector valued form $\mathcal T$ on $A$. 
Similarly, $(p,1)$-tensors of degree $0$ on $A[1]$ identify with linear $(p,1)$-tensors on $A$.

The next theorem translates in terms of $\mathcal T^{[1]}$ the property of $(\mathcal D, l, \mathcal T^M)$ of being infinitesimally multiplicative.

\begin{theorem}\label{theor:Q-forms}
Let $A \Rightarrow M$ be a Lie algebroid, and let $(A[1], d_A)$ be the associated $\mathbb N Q$-manifold of degree $1$. Let $\mathcal T$ be a linear vector valued form on $A$ with components $(\mathcal D, l, \mathcal T^M)$. The assignment $(\mathcal D, l, \mathcal T^M) \mapsto \mathcal T^{[1]}$ establishes a one-to-one correspondence between IM vector valued forms on $A$ and degree $0$ vector valued forms $\mathcal T^{[1]}$ on $A[1]$ such that
$
\mathcal L_{d_A} \mathcal T^{[1]} = 0.
$
\end{theorem}

\begin{proof}
We provide a proof for $2$-forms. For general $p$-forms the proof is similar and it is left to the reader. So let $\mathcal T \in \Omega^2 (A, TA)$ be a linear form, let $(\mathcal D, l, \mathcal T^M)$ be the associated components, and let $\mathcal T^{[1]} \in \Omega^2 (A[1], TA[1])$ be the degree $0$ form corresponding to $\mathcal T$. The Lie derivative $\mathcal L_{d_A} \mathcal T^{[1]}$ is a degree $1$ vector valued form. As vector fields on $A[1]$ are generated in degree $-1$ and $0$, $\mathcal L_{d_A} \mathcal T^{[1]}$ vanishes iff
\begin{align}
\mathcal L_{d_A} \mathcal T^{[1]} (a^\uparrow, b^\uparrow) & = 0, \label{eq:Q-forms1}\\
\mathcal L_{d_A} \mathcal T^{[1]} (a^\uparrow, X_D) & = 0, \label{eq:Q-forms2}\\
\mathcal L_{d_A} \mathcal T^{[1]} (X_{D_1}, X_{D_2}) & = 0, \label{eq:Q-forms3}
\end{align}
for all $a, b \in \Gamma(A)$, and all $D, D_1, D_2 \in \Gamma (DA)$. We will show that Equation (\ref{eq:Q-forms1}) is equivalent to (\ref{eq:IM_tensor_4}), Equation (\ref{eq:Q-forms2}) is equivalent to (\ref{eq:IM_tensor_2}) modulo (\ref{eq:IM_tensor_4}), and Equation (\ref{eq:Q-forms2}) is equivalent to (\ref{eq:IM_tensor_1}) modulo (\ref{eq:IM_tensor_2}) and (\ref{eq:IM_tensor_4}). In view of Remark \ref{rem:relations}, this will conclude the proof. To begin with, a direct computation shows that
\[
\mathcal L_{d_A} \mathcal T^{[1]} (a^\uparrow, b^\uparrow)  = \left(\iota_{\rho_A (a)} l(b)+ \iota_{\rho_A (b)} l(a)\right)^\uparrow,
\]
showing that Equation (\ref{eq:Q-forms1}) is equivalent to (\ref{eq:IM_tensor_4}). Next we assume (\ref{eq:Q-forms1}) and compute $\mathcal L_{d_A} \mathcal T^{[1]} (a^\uparrow, X_D)$. As $\mathcal L_{d_A} \mathcal T^{[1]} (a^\uparrow, X_D)$ is a degree $0$ vector field, it is completely determined by the commutators
\[
\left[\mathcal L_{d_A} \mathcal T^{[1]} (a^\uparrow, X_D), b^\uparrow \right]
\]
for all $b \in \Gamma (A)$. It is easy to see that, after using (\ref{eq:IM_tensor_4}),
we have
\[
\left[\mathcal L_{d_A} \mathcal T^{[1]} (a^\uparrow, X_D), b^\uparrow \right] = \iota_{\sigma (D)} \left ( \mathcal L_b l(a) + \iota_{\rho_A (a)} \mathcal D (b) - l ([a,b]_A) \right)^\uparrow.
\]
As $D$ is arbitrary, we see that (\ref{eq:Q-forms2}) is equivalent to (\ref{eq:IM_tensor_2}) modulo (\ref{eq:IM_tensor_4}).

Finally, we have to compute $\mathcal L_{d_A} \mathcal T^{[1]} (X_{D_1}, X_{D_2})$. It is a degree $1$ vector field, and it is completely determined by the nested commutators 
\[
\left[ \left[\mathcal L_{d_A} \mathcal T^{[1]} (X_{D_1}, X_{D_2}), a^\uparrow \right], b^\uparrow \right].
\]
Similarly as above
\[
\begin{aligned}
&\left[ \left[\mathcal L_{d_A} \mathcal T^{[1]} (X_{D_1}, X_{D_2}), a^\uparrow \right], b^\uparrow \right] \\
&= \iota_{\sigma (D_2)} \iota_{\sigma (D_1)} \big(\mathcal D ([a,b]_A) - \mathcal L_a \mathcal D (b) + \mathcal L_b \mathcal D (a)  \big)^\uparrow \ \text{modulo (\ref{eq:IM_tensor_2}) and (\ref{eq:IM_tensor_4})}.
\end{aligned}
\]
As $D_1, D_2$ are arbitrary, this concludes the proof.
\end{proof}

We now pass to connections. Linear connections in graded geometry are defined exactly as in the ordinary case (up to the appropriate Koszul signs). It is always assumed that the covariant derivative $\nabla_X$ along a vector field $X$ of a definite degree $k$ is an operator of definite degree. As the symbol of $\nabla_X$ is $X$ itself (by definition of linear connection) it follows that the degree of $\nabla_X$ is $k$ as well.

Similarly as for vector valued forms, linear connections in the tangent bundle $T A[1] \to A[1]$ of a degree $1$ $\mathbb N Q$-manifold identify with fiber-wise linear connections in $TA \to A$. Namely, the same formulas as for fiber-wise linear connections identify a connection in $T A[1] \to A[1]$ with a $4$-tuple $(\mathcal F, \nabla^A, \nabla^M, l)$ of components, hence with a fiber-wise linear connection $\nabla$ in $TA \to A$.  We will denote by $\nabla^{[1]}$ the connection in $TA[1] \to A[1]$ corresponding to a fiber-wise linear connection $\nabla$ in $TA \to A$.
Notice that the torsion of $\nabla^{[1]}$ agrees with $(T^\nabla)^{[1]}$: the vector valued $2$-form on $A[1]$ corresponding to the torsion $T^\nabla$ of $\nabla$. 

\begin{theorem}\label{theor:Q-conn}
Let $A \Rightarrow M$ be a Lie algebroid, ad let $(A[1], d_A)$ be the associated $\mathbb N Q$-manifold of degree $1$. Let $\nabla$ be a fiber-wise linear connection in $TA \to A$ with components $(\mathcal F, \nabla^A, \nabla^M, l)$. The assignment $(\mathcal F, \nabla^A, \nabla^M, l) \mapsto \nabla^{[1]}$ establishes a one-to-one correspondence between IM connections in $TA \to A$ and connections $\nabla^{[1]}$ in $TA[1] \to A[1]$ such that $\mathcal L_{d_A} \nabla^{[1]} = 0$.
\end{theorem}

\begin{proof}
The proof of the present theorem
is very similar to that of Theorem \ref{theor:Q-forms}.
The Lie derivative $\mathcal L_{d_A} \nabla^{[1]}$ is a degree $(2,1)$-tensor of degree $1$. It vanishes iff
\begin{align}
\mathcal L_{d_A} \nabla^{[1]} (a^\uparrow, b^\uparrow) & = 0, \label{eq:Q-conn1}\\
\mathcal L_{d_A} \nabla^{[1]} (X_D, a^\uparrow) & = 0, \label{eq:Q-conn2}\\
\mathcal L_{d_A} \nabla^{[1]} (a^\uparrow, X_D) & = 0, \label{eq:Q-conn3}\\
\mathcal L_{d_A} \nabla^{[1]} (X_{D_1}, X_{D_2}) & = 0, \label{eq:Q-conn4}
\end{align}
for all $a, b \in \Gamma(A)$, and all $D, D_1, D_2 \in \Gamma (DA)$.
An easy computation shows that
\[
\mathcal L_{d_A} \nabla^{[1]} (a^\uparrow, b^\uparrow)  = \left([a,b]_A - \nabla^A_{\rho_A(a)} b + \nabla^A_{\rho_A(b)} a + \iota_{\rho_A (b)} l(a)\right)^\uparrow,
\]
whence Equation (\ref{eq:Q-conn1}) is equivalent to (\ref{eq:IM_conn_1}). Next, as $\mathcal L_{d_A} \nabla^{[1]} (X_D, a^\uparrow)$ is a degree $0$ vector field, it is completely determined by the commutators
\[
\left[\mathcal L_{d_A} \nabla^{[1]} (X_D, a^\uparrow), b^\uparrow \right]
\]
for all $b \in \Gamma (A)$, and a direct computation shows that
\[
\left[\mathcal L_{d_A} \nabla^{[1]} (X_D, a^\uparrow), b^\uparrow \right] = \left(\mathcal L_{b}\nabla^A (\sigma (D), a) - \mathcal F(b) (\sigma (D), \rho_A(a))\right)^\uparrow\ \text{modulo (\ref{eq:IM_conn_1})}.
\]
It follows from the arbitrariness of $D$, that (\ref{eq:Q-conn2}) is equivalent to (\ref{eq:IM_conn_2}) modulo (\ref{eq:IM_conn_1}). Similarly,
\[
\left[\mathcal L_{d_A} \nabla^{[1]} (a^\uparrow, X_D), b^\uparrow \right]= \iota_{\sigma (D)}\big(l ([b,a]_A) - \mathcal L_b l (a) + \iota_{\rho_A(a)} \mathcal D (b)\big)^\uparrow\ \text{modulo (\ref{eq:IM_conn_1}) and (\ref{eq:IM_conn_2})}.
\]
Hence, (\ref{eq:Q-conn3}) is equivalent to (\ref{eq:IM_conn_4}) modulo (\ref{eq:IM_conn_1}) and (\ref{eq:IM_conn_2}). Finally, $\mathcal L_{d_A} \nabla^{[1]} (X_{D_1}, X_{D_2}) $ is a degree $1$ vector field, and it is completely determined by the nested commutators 
\[
\left[ \left[\mathcal L_{d_A} \mathcal T^{[1]} (X_{D_1}, X_{D_2}), a^\uparrow \right], b^\uparrow \right].
\]
A long but easy computation now shows that
\[
\begin{aligned}
&\left[ \left[\mathcal L_{d_A} \mathcal T^{[1]} (X_{D_1}, X_{D_2}), a^\uparrow \right], b^\uparrow \right] \\
&= \iota_{\sigma (D_2)} \iota_{\sigma (D_1)} \big(\mathcal F ([a,b]_A) - \mathcal L_a \mathcal F (b) + \mathcal L_b \mathcal F (a)  \big)^\uparrow \ \text{modulo  (\ref{eq:IM_conn_1}), (\ref{eq:IM_conn_2}) and (\ref{eq:IM_conn_4})}.
\end{aligned}
\]
As $D_1, D_2$ are arbitrary, this concludes the proof.
\end{proof}

Theorem \ref{theor:Q-conn} allows us to identify the exact obstruction to the existence of IM connections in the tangent bundle of a Lie algebroid. Namely, let $A \Rightarrow M$ be a Lie algebroid. The graded space $\mathfrak T^{2,1} (A[1])$ of $(2,1)$ tensors on $A[1]$ is a cochain complex when equipped with the Lie derivative $\mathcal L_{d_A}$ along the $Q$-structure $d_A$, and we have the following infinitesimal version of Theorem \ref{theor:obst_class}:

\begin{theorem}[Mehta, Sti\'enon, Xu \cite{MSX2015}]\label{theor:inf_obst_class}
Let $A \Rightarrow M$ be a Lie algebroid and pick any fiber-wise linear connection $\nabla$ in $TA$. Consider the $(2,1)$ tensor $\mathsf{At}_\nabla := \mathcal L_{d_A} \nabla^{[1]} \in \mathfrak T^{2,1} (A[1])$. Then
\begin{enumerate}
\item $\mathsf{At}_\nabla$ is a $\mathcal L_{d_A}$-cocycle, i.e. $\mathcal L_{d_A}\mathsf{At}_\nabla = 0$;
\item The cohomology class $[\mathsf{At}_\nabla] \in H^1 (\mathfrak T^{2,1} (A[1]), \mathcal L_{d_A})$ is independent of the choice of $\nabla$;
\item The cohomology class $[\mathsf{At}_\nabla]$ vanishes iff there exists an IM connection in $TA$.
\end{enumerate}
\end{theorem}

\begin{proof}
The proof is discussed in a more general setting in \cite{MSX2015}. In that paper,  $[\mathsf{At}_\nabla]$ is called the \emph{Atiyah class} of the $Q$-manifold $(A[1], d_A)$. This motivates our notation.
\end{proof}

We are now ready to prove Theorem (\ref{theor:IM_spray}). 

\begin{proof}[Proof of Theorem \ref{theor:IM_spray}]
Let $\nabla$ be a fiber-wise linear connection in the tangent bundle $TA \to A$ of a Lie algebroid $A \Rightarrow M$. 
First we show that (1) $\Rightarrow$ (2). 
We do that showing that (2) is equivalent to $\mathcal L_{d_A} \nabla^{[1]} = 0$, and then using Theorem \ref{theor:Q-conn}. It is enough to assume that $\nabla$ is a symmetric (i.e. torsion free) connection. We need a premise on the geodesic spray of a connection in the tangent bundle of a graded manifold.

First of all, let $N$ be an ordinary (non-graded) manifold, and let $\nabla$ be a connection in $TN \to N$. The geodesic spray $Z^\nabla \in \mathfrak X(TN)$ is a vector field of degree $1$ wrt the action $h$ of $(\mathbb R, \cdot)$ on $TN$ given by the fiber-wise scalar multiplication, i.e. $h_t^\ast Z^\nabla = tZ^\nabla$, for all $t \neq 0$, and it is completely determined by the following identity:
\begin{equation}\label{eq:gspray, alg}
[[Z^\nabla, X^\uparrow], Y^\uparrow] = - \left(\nabla_X Y + \nabla_Y X\right)^\uparrow
\end{equation}
for all $X,Y \in \mathfrak X (M)$, where $X^\uparrow \in \mathfrak X (TN)$ is the vertical lift of $X$. We can also relate the Lie derivative of $\nabla$ along a vector field $Q \in \mathfrak X (N)$ with the commutator $[Q^{\mathrm{tan}}, Z^\nabla]$, where $Q^{\mathrm{tan}}$ is the tangent lift of $Q$. Namely, it follows from (\ref{eq:gspray, alg}) and the identity $[Q^{\mathrm{tan}}, X^\uparrow] = [Q, X]^\uparrow$ that
\begin{equation}\label{eq:[Q,G]}
[[Q^{\mathrm{tan}}, Z^\nabla], X^\uparrow], Y^\uparrow] = - \left(\mathcal L_Q \nabla (X, Y) + \mathcal L_Q \nabla (Y,X) \right)^\uparrow
\end{equation}
for all $X,Y \in \mathfrak X (M)$. Equation (\ref{eq:[Q,G]}) shows, in particular, that $\mathcal L_{Q}\nabla^{\mathrm{sym}} = 0$ iff $[Q^{\mathrm{tan}}, Z^\nabla] = 0$.

Formula (\ref{eq:gspray, alg}) allows us to define a geodesic spray for a linear connection in the tangent bundle of any graded manifold. Namely, we have the following
\begin{lemma}
Let $\mathcal M$ be a graded manifold, let $h : \mathbb R \times T \mathcal M \to T \mathcal M$ be the scalar multiplication in the fibers of $T\mathcal M \to \mathcal M$, and let $\nabla$ be a connection in $T \mathcal M \to \mathcal M$. There exists a unique vector field $Z^\nabla$ on $T \mathcal M$ with the following properties:
\begin{enumerate}
\item $Z^\nabla$ has internal (coordinate) degree $0$,
\item $h_t^\ast Z^\nabla = t Z^\nabla$ for all $t \neq 0$,
\item Equation (\ref{eq:gspray, alg}) holds for all $X, Y \in \mathfrak X (\mathcal M)$, where $X^\uparrow, Y^\uparrow \in \mathfrak X (T\mathcal M)$ are the vertical lifts of $X,Y$.
\end{enumerate}
Formula (\ref{eq:[Q,G]}) is still valid in this graded setting.
\end{lemma}

Now, let $A \to M$ be a vector bundle, and let $\nabla$ be a fiber-wise linear connection in $TA \to A$. Consider the connection $\nabla^{[1]}$ in $TA[1] \to A[1]$ corresponding to $\nabla$, and the associated \emph{geodesic spray} $Z^{\nabla^{[1]}}$. Recall that $TA[1]$ is canonically isomorphic to the $\mathbb N$-manifold of degree $1$ obtained by shifting by one the degree in the fibers of the vector bundle $TA \to TM$, and, in the following, we will understand this isomorphism (see, e.g., \cite[Lemma 2.3.1]{LV2019} for details). As the internal degree of $Z^{\nabla^{[1]}}$ is $0$, it corresponds to a derivation of $TA \to TM$, and to a vector field on $TA$ (more precisely, a linear vector field wrt the vector bundle structure $TA \to TM$). One can show, working for instance in local coordinates, that the latter vector field is exactly $Z^\nabla$. Namely, 
from $\nabla$ being fiber-wise linear it easily follows that $Z^\nabla$ is linear wrt the vector bundle structure $TA \to TM$, and
\[
(Z^\nabla)^{[1]} = Z^{\nabla^{[1]}}.
\]
Finally, let $A \Rightarrow M$ be a Lie algebroid and let $\nabla$, hence $\nabla^{[1]}$, be a symmetric fiber-wise linear connection in $TA \to A$. Recall that $TA$ is equipped with the tangent prolongation Lie algebroid structure $TA \Rightarrow TM$, and the corresponding homological vector field on $TA[1]$ is exactly the tangent lift $d_A^{\mathrm{tan}}$. Collecting all the above remarks we get the following chain of equivalences: the components of $\nabla$ form an IM connection iff $\mathcal L_{d_A} \nabla^{[1]} = 0$ iff $[d_A^{\mathrm{tan}}, Z^{\nabla^{[1]}}] = [d_A^{\mathrm{tan}}, (Z^{\nabla})^{[1]}] = 0$ iff $Z^\nabla$ is an IM vector field wrt the Lie algebroid structure $TA \Rightarrow TM$. This concludes the proof.
\end{proof}

\bigskip

\textbf{Acknowledgements.} LV is member of the GNSAGA of INdAM. We thank I. Struchiner for useful discussions. We also thank R. Fernandes and C. Ortiz for sharing some details of their work in progress on \emph{connections on principal bundles in the category of Lie groupoids}. We look forward to comparing our theory with theirs.

\end{document}